\documentclass{amsart}
\usepackage{amsthm,amssymb,amsmath,color}
\usepackage{pdfsync}

\newcommand{\red}{\color{darkred}}
\newcommand{\blue}{\color{darkblue}}

\newcommand{\levy}{L\'evy}

\newcommand{\clt}{central limit theorem}

\newcommand{\garch}{{\rm GARCH}$(1,1)$}

\newcommand{\sta}{St\u aric\u a}
\newcommand{\ex}{{\rm e}\,}

\newcommand{\asy}{asymptotic}

\newcommand{\ts}{time series}

\definecolor{darkblue}{rgb}{.1, 0.1,.8}
\definecolor{darkgreen}{rgb}{0,0.8,0.2}
\definecolor{darkred}{rgb}{.8, .1,.1}

\def\tag{\refstepcounter{equation}\leqno }

\newtheorem{lemma}{Lemma}[section]

\newtheorem{theorem}[lemma]{Theorem}

\newtheorem{proposition}[lemma]{Proposition}
\newtheorem{definition}[lemma]{Definition}
\newtheorem{corollary}[lemma]{Corollary}
\newtheorem{example}[lemma]{Example}
\newtheorem{exercise}[lemma]{Exercise}
\newtheorem{remark}[lemma]{Remark}
\newtheorem{fig}[lemma]{Figure}
\newtheorem{tab}[lemma]{Table}

\newcommand{\MC}{Markov chain}
\newcommand{\bfQ}{{\bf Q}}

\newcommand{\bfR}{{\bf R}}

\newcommand{\bth}{\begin{theorem}}
\newcommand{\ethe}{\end{theorem}}

\newcommand{\bre}{\begin{remark}\em }
\newcommand{\ere}{\end{remark}}

\newcommand{\ble}{\begin{lemma}}
\newcommand{\ele}{\end{lemma}}
\newcommand{\sre}{stochastic recurrence equation}
\newcommand{\pp}{point process}
\newcommand{\bde}{\begin{definition}}
\newcommand{\ede}{\end{definition}}
\newcommand{\bco}{\begin{corollary}}
\newcommand{\eco}{\end{corollary}}

\newcommand{\bpr}{\begin{proposition}}
\newcommand{\epr}{\end{proposition}}

\newcommand{\bexer}{\begin{exercise}}
\newcommand{\eexer}{\end{exercise}}

\newcommand{\bexam}{\begin{example}}
\newcommand{\eexam}{\end{example}}

\newcommand{\bfi}{\begin{fig}}
\newcommand{\efi}{\end{fig}}

\newcommand{\btab}{\begin{tab}}
\newcommand{\etab}{\end{tab}}

\newcommand{\lhs}{left-hand side}
\newcommand{\fidi}{finite-dimensional distribution}
\newcommand{\rv}{random variable}

\newcommand{\sign}{{\rm sign}}

\newcommand{\var}{{\rm var}}

\newcommand{\cov}{{\rm cov}}

\newcommand{\rhs}{right-hand side}

\newcommand{\beao}{\begin{eqnarray*}}
\newcommand{\eeao}{\end{eqnarray*}\noindent}

\newcommand{\beam}{\begin{eqnarray}}
\newcommand{\eeam}{\end{eqnarray}\noindent}

\newcommand{\beqq}{\begin{equation}}
\newcommand{\eeqq}{\end{equation}\noindent}

\newcommand{\bce}{\begin{center}}
\newcommand{\ece}{\end{center}}

\newcommand{\barr}{\begin{array}}
\newcommand{\earr}{\end{array}}

\newcommand{\stp}{\stackrel{P}{\rightarrow}}
\newcommand{\std}{\stackrel{d}{\rightarrow}}
\newcommand{\stas}{\stackrel{\rm a.s.}{\rightarrow}}

\newcommand{\stv}{\stackrel{v}{\rightarrow}}
\newcommand{\stw}{\stackrel{w}{\rightarrow}}

\newcommand{\eqd}{\stackrel{d}{=}}

\newcommand{\vague}{\stackrel{\lower0.2ex\hbox{$\scriptscriptstyle
                    \it{v} $}}{\rightarrow}}
\newcommand{\weak}{\stackrel{\lower0.2ex\hbox{$\scriptscriptstyle
                    \it{w} $}}{\rightarrow}}
\newcommand{\what}{\stackrel{\lower0.2ex\hbox{$\scriptscriptstyle
                    \it{\hat{w}} $}}{\rightarrow}}

\newcommand{\bdis}{\begin{displaymath}}
\newcommand{\edis}{\end{displaymath}\noindent}

\newcommand{\N}{\mathbb{N}}
\newcommand{\R}{\mathbb{R}}

\newcommand{\nto}{n\to\infty}
\newcommand{\kto}{k\to\infty}
\newcommand{\xto}{x\to\infty}

\newcommand{\ov}{\overline}
\newcommand{\wt}{\widetilde}

\newcommand{\vep}{\varepsilon}

\newcommand{\la}{\lambda}

\newcommand{\regvary}{regularly varying}

\newcommand{\regvar}{regular variation}

\newcommand{\bbr}{{\mathbb R}}

\newcommand{\bbz}{{\mathbb Z}}
\newcommand{\bbn}{{\mathbb N}}

\newcommand{\bbs}{{\mathbb S}}

\newcommand{\bfE}{{\mathbf E}}

\newcommand{\con}{convergence}

\newcommand{\st}{such that}
\newcommand{\fif}{if and only if}
\newcommand{\wrt}{with respect to}
\newcommand{\chf}{characteristic function}
\newcommand{\fct}{function}

\newcommand{\ds}{distribution}

\newcommand{\rep}{representation}
\newcommand{\cmt}{continuous mapping theorem}
\newcommand{\seq}{sequence}

\newcommand{\ms}{measure}

\newcommand{\ld}{large deviation}

\newcommand{\bfX}{{\bf X}}
\newcommand{\bfD}{{\bf D}}

\newcommand{\bfY}{{\bf Y}}
\newcommand{\bfy}{{\bf y}}
\newcommand{\bfA}{{\bf A}}

\def\1{\ensuremath{\mathrm{1}\hspace{-.35em} \mathrm{1}}} 
\def\B{\mathbb{B}}

\def\E{{\mathbb E}}

\def\L{\mathbb{L}}
\def\N{\mathbb{N}}
\def\P{{\mathbb{P}}}
\def\R{\mathbb{R}}
\def\Z{\mathbb{Z}}
\def\Lip{\mathop{\rm Lip}\nolimits}

\renewcommand{\le}{\ensuremath{\leqslant}}
\renewcommand{\ge}{\ensuremath{\geqslant}}

\newcommand{\introo}[2]{{\left]{#1,\,#2\,}\right[\kern1pt}}

\newcommand{\intrfo}[2]{{\left[{#1,\,#2}\right[\kern1pt}}

\begin{document}

\title[The cluster index]{The cluster index of regularly varying sequences
with applications to limit theory for functions of multivariate Markov chains}

\thanks{Thomas Mikosch's research is partly supported
by the Danish Research Council (FNU) Grants 272-06-0442 and
09-072331. 
Both authors would like to thank their home institutions for
hospitality when visiting each other.}
\author[T. Mikosch]{Thomas Mikosch}
\author[O. Wintenberger]{Olivier Wintenberger}
\address{Thomas Mikosch, University of Copenhagen, Department of Mathematics,
Universitetsparken 5,
DK-2100 Copenhagen\\ Denmark} \email{mikosch@math.ku.dk}

\address{Olivier Wintenberger,
            Universit\'e de Paris-Dauphine and CREST-LFA,
            Centre De Recherche en Math\'ematiques de la D\'ecision
            UMR CNRS 7534,
            Place du Mar\'echal De Lattre De Tassigny,
            75775 Paris Cedex 16, France}
            \email{owintenb@ceremade.dauphine.fr}

\maketitle

\begin{abstract}{  We introduce the {\em cluster index}
of a multivariate \regvary\ stationary \seq\ and characterize the
index in terms of the spectral tail process. This index plays a major
role in limit theory for partial sums of \regvary\ \seq s. We illustrate the use of the 
cluster index by characterizing infinite variance stable
limit \ds s and precise \ld\ results for sums of  multivariate \fct s 
acting on a stationary \MC\ under a drift condition.}
\end{abstract}



\section{Introduction}
Consider  a stationary \MC\ $(\Phi_t)$ and a \fct\
$f$ acting on the state space of the \MC\ and mapping into $\R^d$ for some
$d\ge 1$. For the resulting stationary process $X_t=f(\Phi_t)$, $t\in\bbz$,
the corresponding partial sum process is given by
\beao
S_0=0\,,\quad S_n= X_1+\cdots +X_n\,.
\eeao
We also assume that the \fidi s of the process $(X_t)$ are \regvary\
with index $\alpha$; see Section~\ref{subsec:regvar} for a definition.
Roughly speaking, this condition ensures that the tails of the \fidi s
have power law behavior, hence sufficiently high moments of $X$ are infinite.
(Here and in what follows, we write $Y$ for a generic element of any 
stationary \seq\ $(Y_t)$.) Regular variation of a random vector and,
more generally, of a stationary \seq\ is a condition which determines
the extremal dependence structure in a flexible way.
\par
For an iid \seq\ the condition of \regvar\ of $X$ with index
$\alpha\in (0,2)$ is necessary and sufficient for the \clt\
\beao
a_n^{-1} (S_n-b_n)\std \xi_\alpha\,,\quad \nto ,
\eeao
where $a_n>0$, $b_n\in\bbr$, $n\in\bbn$, are suitable constants and
$\xi_\alpha$ has an $\alpha$-stable \ds\ in $\bbr^d$; see
\cite{rvaceva:1962} for the limit theorem and
\cite{samorodnitsky:taqqu:1994}
for a description of infinite variance stable laws in $\bbr^d$.
Limit theory with $\alpha$-stable limits for dependent \seq s
was studied in  \cite{jakubowski:1993,jakubowski:1997} by using 
the \con\ of \chf s and in \cite{davis:hsing:1995} by using the \cmt\
acting on suitable weakly converging \pp es; see also 
\cite{basrak:krizmanic:segers:2011} for a \fct al \clt\ using the same
technique.
These results were
proved for univariate \seq s, but \cite{davis:mikosch:1998} proved
that the \pp\ \con\ results remain valid in the multivariate case by
a slight modification of the proofs in  \cite{davis:hsing:1995}.
\par
Using ideas from \cite{jakubowski:1993,jakubowski:1997}, the authors of
\cite{bartkiewicz:jakubowski:mikosch:wintenberger:2011} studied 
stable limit theory for general univariate \regvary\ \seq s; see
Theorem~\ref{thm:clt} below. We use this result and the Cram\'er-Wold
device to derive the corresponding limits for all linear combinations
$\theta'S_n$, $\theta\in \bbs^{d-1}$, where  $\bbs^{d-1}$ is the unit
sphere in $\bbr^d$ \wrt\ the Euclidean norm. According to
Theorem~\ref{thm:clt}, the $\alpha$-stable limit laws of $\theta'S_n$
(with suitable normalization and centering) are characterized by
the \fct\
\beam\label{eq:oo}
b(\theta)=
\lim_{k\to\infty}\lim_{\xto}
\dfrac{\P(\theta'S_k>x)-\P(\theta'S_{k-1}>x)}{\P(|X|>x)}  \,,\quad \theta\in
\bbs^{d-1}\,.
\eeam
We discuss the so-called {\em cluster index} $b$ 
in Section~\ref{subsec:cluster}. The existence of the limits
in \eqref{eq:oo} is guaranteed under the conditions of this paper; see 
Theorem~\ref{lem:sum}. Moreover, the cluster index $b$
determines the $\alpha$-stable limit laws in the multivariate case;
see Theorem~\ref{main}. In a way, the \fct\ $b$ plays a similar 
role as the notion of extremal index in limit theory for maxima of
dependent \seq s; see \cite{leadbetter:lindgren:rootzen:1983} for this
notion.
\par
Regular variation is also the key to {\em precise \ld\ theory} for the
sums $S_n$. In the univariate iid case, classical work by 
A.V. and S.V. Nagaev \cite{nagaev:1969,nagaev:1979} shows that
relations of the following type hold
\beao
\sup_{x\ge b_n}\Big|\dfrac{\P(S_n>x)}{n\,\P(|X|>x)}- p\Big|\to 0\,,
\eeao
where $p=\lim_{\xto}\P(X>x)/\P(|X|>x)=p$ and $(b_n)$ is a suitably
chosen \seq\ \st\ $b_n\to\infty$ and { $S_n/b_n\stp 0$ as $\nto$.} Related work for dependent \regvary\
\seq s was proved in \cite{mikosch:samorodnitsky:2000} for linear 
processes, 
in \cite{buraczewski:damek:mikosch:zienkiewicz:2011, konstantinides:mikosch:2004} for solutions to
\sre s
and for general \regvary\ \seq s in \cite{mikosch:wintenberger:2011};
for earlier results see also 
\cite{jakubowski:1993,jakubowski:1997,davis:hsing:1995}.   
The results in these papers are all of the type 
\beam\label{eq:mmm}
\sup_{x\in (b_n,c_n)}\Big|\dfrac{\P(S_n>x)}{n\,\P(|X|>x)}- b(1)\Big|\to 0\,,
\eeam
where $b(1)$ is the limit in \eqref{eq:oo} for $d=1$ and $(b_n,c_n)$
are suitable regions tending to infinity. 
\par
The case of iid multivariate \regvary\ $(X_t)$ was treated in 
\cite{hult:lindskog:mikosch:samorodnitsky:2005}, including a
corresponding \fct al \ld\ result.   In this paper, we get a 
corresponding \ld\ principle  for \regvary\  multivariate \fct s acting on a 
\MC\   (see 
Theorem~\ref{thm:ldp}): 
\beam\label{eq:gg}
\dfrac{\P (\lambda_n^{-1}S_n\in\cdot)}{n\,\P(|X|>\lambda_n)}\stv 
\nu_\alpha\,.
\eeam
Here $\stv$ denotes vague \con\ on some Borel $\sigma$-field,
$\la_n\to \infty$ is a suitable normalizing \seq<  
and the limit $\nu_\alpha$ is a \ms\  which is
induced by the \regvar\ of the sums $S_k$, $k\ge 1$. 
As for the case of stable limits, we
start by proving the \ld\ principle for linear 
combinations $\theta' S_n$, exploiting the corresponding result
\eqref{eq:mmm} with $b(1)$ replaced by $b(\theta)$ from \eqref{eq:oo}; 
see Theorem~\ref{ldMW}. This corresponds to  
\eqref{eq:gg} restricted to half-planes not containing the origin.  
It is in general not possible to extend the limit relation 
\eqref{eq:gg} from half-spaces to general Borel sets. This 
extension is however possible 
by assuming some additional conditions such as $\alpha$ is non-integer.
As a matter of fact, relation \eqref{eq:gg} cannot be written as a
uniform result in the spirit of \eqref{eq:mmm}, due to its
multivariate character.
\par 
The paper is  organized as follows. In Section~\ref{sec:prel}
we introduce \regvar\ of a stationary \seq\ and the drift condition
of a \MC . In Section~\ref{subsec:cluster} we define 
the {\em cluster index} $b(\theta)$, $\theta\in \bbs^{d-1}$, of a
\regvary\ stationary \seq . We prove the existence of the cluster
index for multivariate \fct s acting on a \MC\ under a drift
condition (Theorem~\ref{lem:sum}).
In Section~\ref{sec:stab} we formulate the main results of this paper.
They include $\alpha$-stable limit theory (Theorem~\ref{main}) 
and precise \ld\ principles (Theorem~\ref{sec:ldp}) 
for \fct s of regenerative \MC s. In Section~\ref{sec:ex}
we calculate the cluster index for
several important \ts\ models, including multivariate autoregressive
processes, solutions to \sre s, \garch\ processes and their sample
covariance \fct s.
In the remaining sections we prove the results of  Section~\ref{sec:stab}.

\section{Preliminaries}\label{sec:prel}\setcounter{equation}{0}
\subsection{Regular variation of vectors and \seq s of random  vectors}\label{subsec:regvar}
In what follows, we will use the notion of \regvar\ as a suitable
way of describing heavy tails of random vectors and \seq s of random vectors.
We commence with a random vector $X$ with values in $\R^d$ for some $d\ge 1$.
We say that this vector (and its \ds ) are {\em \regvary\ with index
$\alpha >0$} if the following relation holds as $\xto$:
\begin{equation}\label{eq:rv}
\frac{\P(|X| > ux, X/| X| \in\cdot )}{\P(|X| > x)}\stw
u^{-\alpha}\, \P(\Theta\in\cdot),\quad u>0\,.
\end{equation}
Here $\stw$ denotes weak \con\ of finite \ms s  and $\Theta$ is a vector with values in the unit sphere
$\bbs^{d-1} = \{x\in\bbr^d : |x| = 1\}$ of   $\R^d$. Its \ds\ is the {\em spectral \ms }
of \regvar\ and depends on the choice of the norm. However, the definition of \regvar\ does not
depend on any concrete norm; we always refer to the Euclidean norm. An equivalent way to define
\regvar\ of $X$ is to require that there exists a non-null Radon \ms\ $\mu$ on the
Borel $\sigma$-field of $\ov \bbr_0^d=
\ov \bbr^d\setminus\{\bf0\}$ \st\
\beam\label{eq:1}
n\, \P(a_n^{-1} X\in \cdot )\stv \mu_X(\cdot)\,,
\eeam
where the \seq\ $(a_n)$ can be chosen \st\ $n\,\P(|X|>a_n)\sim 1$ and $\stv$ refers to
vague \con . The limit \ms\ $\mu_X$ necessarily has the property 
$\mu_X(u\cdot)= u^{-\alpha}\mu_X(\cdot)\,,u>0$, 
which explains the relation with the index $\alpha$.
We refer to
\cite{bingham:goldie:teugels:1987} for an encyclopedic treatment of one-dimensional \regvar\ and
\cite{resnick:1987,resnick:2007} for the multivariate case.
\par
Next consider a strictly stationary \seq\ 
$(X_t)_{t\in\Z}$ of $\bbr^d$-valued random vectors with a generic
element $X$.
It is {\em \regvary\ with index $\alpha>0$} if every lagged vector
$(X_1, . . . ,X_k)$, $k\ge 1$, is \regvary\ in the sense of \eqref{eq:rv}; see
\cite{davis:hsing:1995}. An equivalent description of a \regvary\ \seq\ $(X_t)$ is achieved by
exploiting \eqref{eq:1}: for every $k\ge 1$, there exists a non-null Radon \ms\  $\mu_k$ on
 the Borel $\sigma$-field of $ \ov \bbr_0^{dk}$ \st\
\beam\label{eq:1a}
n\, \P(a_n^{-1} (X_1,\ldots,X_k)\in \cdot )\stv \mu_k\,,
\eeam
where $(a_n)$ is chosen \st\ $n\,\P(|X_0|>a_n)\sim 1$.
\par
A convenient characterization of a \regvary\ \seq\  $(X_t)$ was given in
Theorem 2.1 of \cite{basrak:segers:2009}: there exists a \seq\ of $\bbr^d$-valued random vectors
$(Y_t)_{t\in\bbz}$ \st\ $\P(|Y_0| > y) = y^{-\alpha}$ for $y > 1$ and
for $k\ge 0$,
\beao
\P (x^{-1}(X_{-k},\ldots,X_k)\in\cdot \mid |X_0| > x)\stw
\P((Y_{-k},\ldots,Y_k)\in \cdot)\,,\quad \xto\,.
\eeao
The process
$(Y_t)$ is the {\em tail process} of $(X_t)$. Writing $\Theta_t = Y_t/|Y_0|$ for $t\in \Z$,
one also has for $k\ge 0$,
\beam\label{eq:sepctraltailprocess}
\P( |X_0|^{-1}(X_{-k},\ldots,X_k)\in\cdot \mid |X_0| > x)
\stw \P ((\Theta_{-k},\ldots,\Theta_k)\in \cdot)\,,\quad \xto\,.
\eeam
 We  will identify 
$|Y_0|\,(Y_t/|Y_0|)_{|t|\le k}= |Y_0|\,(\Theta_t)_{|t|\le k}$, $k\ge
0$. Then $|Y_0|$ is independent of $(\Theta_t)_{|t|\le k}$ for every
$k\ge 0$. We refer
to $(\Theta_t)_{t\in\bbz}$ as 
the {\em spectral tail process} of $(X_t)$.
\par
We formulate our
main condition on the tails of the \seq\ $(X_t)$: \\[1mm]
{\em Condition} ${\bf (RV_\alpha)}$: The strictly stationary \seq\ $(X_t)$ is
\regvary\ with index $\alpha>0$ and spectral tail process $(\Theta_t)$.

\subsection{The drift condition}\label{subsec:drift}
Assume that the following {\em drift condition} holds for the \MC\ $(\Phi_t)$ for suitable
$p>0$  and an $\bbr^d$-valued  \fct\ $f$ acting on the state space of the \MC :\\[1mm]
{\em Condition}  ${\bf (DC_{\it p})}$:
There exist constants $\beta\in  (0,1)$, $b>0$, and a function $V:\R^d\to(0,\infty)$ such that $c_1|x|^{p}\le V(x)\le c_2|x|^p$, $c_1,c_2>0$, satisfying for any $y$ in the state space of the \MC ,
\beao
\E( V(f(\Phi_1))\mid \Phi_0=y)\le \beta \,V(f(y))+b.
\eeao
We mention that Jensen's inequality ensures that ${\bf (DC_{\it p})}$ implies 
${\bf (DC_{\it p'})}$ for $p'<p$. 
We exploited  condition   ${\bf (DC_{\it p})}$ in
\cite{mikosch:wintenberger:2011}, where we proved \ld\ principles for
\regvary\ strictly stationary \seq s of \rv s, in particular for
irreducible \MC s.
\par
If $(\Phi_t)$ is an irreducible \MC\ then ${\bf (DC_{\it p})}$ for any $p>0$ implies $\beta$-mixing
with geometric rate; see \cite{meyn:tweedie:1993}, p. 371.
Moreover, without loss of generality, by considering the Nummelin splitting scheme, see \cite{nummelin:1984} for details,  we will assume that  $(\Phi_t)$ 
possesses an atom $A$.
The notions of drift, small set, atom, etc. used throughout  are borrowed from
\cite{meyn:tweedie:1993}.   
In what follows, we write $\P_A(\cdot)= \P(\cdot \mid \Phi_0\in A)$ 
and $\E_A$ for the corresponding expectation.
 
\par
We always assume the existence of 
some $M>0$ such that $\{x\,:\, V(f(x))\le M\}$  is a small set 
(this is true in all our examples). Then the condition ${\bf (DC_{\it p})}$ is equivalent to the existence of  
constants $\beta\in  (0,1)$ and $b>0$ such that for any $y$, 
\beao
\E( V(f(\Phi_1)\mid \Phi_0=y)\le \beta \,V(f(y))+b\1_A(y)\,.
\eeao
\par
Direct verification of the condition ${\bf (DC_{\it p})}$ is in general
difficult. We will use the following result which can often be checked
much easier.
\ble\label{lem:skel}
Assume that the stationary \MC\ $(\Phi_{t})$ is aperiodic,  irreducible and 
satisfies the following
condition for some $p>0$ and integer $m\ge 1$ : \\[1mm]
{\em Condition}  ${\bf (DC_{\it p,m})}$: {\rm (a)} There exist $b>0$
and $\beta\in (0,1)$ \st\
for any $y$ in the state space of the \MC ,
\beao
\E( V(f(\Phi_m))\mid \Phi_0=y)\le \beta \,V(f(y))+b\1_A(y)\,,
\eeao 
where $V$ is the \fct\ from ${\bf (DC_{\it p})}$.\\[1mm]
{\rm (b)}
There exist $c_1,c_2>0$  such that for any $y$ in the state space of the \MC 
\beao 
\E( V(f(\Phi_1)\mid \Phi_0=y)\le c_1V(f(y))+c_2\,.
\eeao 
Then condition  ${\bf (DC_{\it p})}$ holds.
\ele
\begin{proof} 
Theorem 15.3.3 in  \cite{meyn:tweedie:1993}
says that the drift condition in part (a) of ${\bf (DC_{\it p,m})}$ 
implies $V$-geometric regularity of the $m$-skeleton \MC\
$(\Phi_{tm})$. 
Theorem 15.3.6 in  \cite{meyn:tweedie:1993} yields 
the equivalence between $V$-geometric regularity and 
$g$-geometric regularity of the original \MC\ for a  
function $g$ satisfying 
$\sum_{t=1}^m\E(g(\Phi_t)\mid \Phi_0=y)=V(f(y))$. Thus the 
drift condition is satisfied for the original \MC\ and some 
finite Lyapunov function $V'\ge g$. Making multiple use of part (b) of 
${\bf (DC_{\it p,m})}$, we can show that there exist 
constants $c_1',c_2'>0$ satisfying $\sum_{t=1}^m\E(g(\Phi_t)\mid \Phi_0=y)\le c_1' V(f(y))+c_2'$. Thus {\bf (DC$_{\it p}$)} 
follows for a \fct\ $V'(x)=c_1''V(x)+c_2''$ and suitable constants
$c_1'',c_2''>0$. 
\end{proof}
Consider the \seq\ of the hitting times of the atom $A$ by
 the \MC\ $(\Phi_t)$, i.e. $\tau_A(1)=\tau_A=\min\{k>0: \Phi_k\in A\}$ and
$\tau_A(j+1)=\min\{k>\tau_A(j): \Phi_k\in A\}$, $j\ge 1$.
We will write 
\beam\label{eq:partial}
S(0)=\sum_{t=1}^{\tau_A}X_t \quad\mbox{and}\quad S(i)=\sum_{t=\tau_A(i)+1}^{\tau_A(i+1)}X_t\,,\quad i\ge 1\,.
\eeam
 According to the theory in \cite{meyn:tweedie:1993},
$(\tau_A(i)-\tau_A(i-1))_{i\ge 2}$ and $(S(i))_{i\ge 1}$ constitute
iid \seq s; we will refer to regenerative \MC s.
The drift condition  {\bf (DC$_{\it p}$)} is tailored 
for proving the existence of moments of $S(1)$
under the existence of moments of $X_t=f(\Phi_t)$ of the same order. 
\par
The drift condition {\bf (DC$_{\it p}$)} is 
useful for proving central limit theory and
other \asy\ results for \fct s of \MC s. As a benchmark result
we quote a \clt\ which is a simple corollary of Proposition 2.1 in
Samur \cite{samur:2004}. To apply this result notice that
${\bf (DC_{\rm 1})}$ implies condition (D$_2$) of \cite{samur:2004}
for $|X_t|$ with $V=c\,|f|$ with $c>0$ sufficiently small.
\bth\label{thm:samur} Assume that the stationary \MC\ $(\Phi_t)$ is 
aperiodic, irreducible and $(X_t)=(f(\Phi_t))$
satisfies ${\bf (DC_{\rm 1})}$, $\E |X|^2<\infty$ and $\E X=0$. 
Then the following statements hold:
\begin{enumerate}
\item
 The partial sum $S(1)$ has finite second moment. 
\item The central limit theorem $n^{-0.5}S_n\std \mathcal N(0,\Sigma)$
  holds with \beao
\Sigma&=& \E_A[S(1)S(1)']\\
&=&  \lim_{k\to \infty} \E\Big[\Big(\sum_{t= 0}^kX_t\Big)\Big(\sum_{t=0}
  ^kX_t\Big)'-\Big(\sum_{t= 1}^k X_t\Big)\Big(\sum_{t=1}^k
X_t\Big)'\Big]\,.
\eeao
\end{enumerate}
\ethe
Together with Theorem~\ref{main} that deals
with the case of infinite variance stable limits, Theorem~\ref{thm:samur} complements 
the limit theory for partial sums of \fct s of \MC s in the case of
finite variance summands and Gaussian limits.

\section{The cluster index }\label{subsec:cluster}\setcounter{equation}{0}
We commence by considering a general $\bbr^d$-valued 
stationary process $(X_t)$ satisfying  ${\bf
  (RV_\alpha)}$  for some $\alpha>0$.
A continuous mapping argument for \regvar\ (see e.g.
\cite{hult:lindskog:2005,hult:lindskog:2006}) and \eqref{eq:1a}
ensure  the existence of the limits
\beao
b_k(\theta) =\lim_{\nto}n\,\P(\theta'S_k >a_n)\,,\quad k\ge 1,\quad  \theta\in \bbs^{d-1}.
\eeao
The difference $b_{k+1}(\theta)-b_k(\theta)$ can be expressed in terms
of the spectral tail process $(\Theta_t)$ of $(X_t)$.
\ble
Let $(X_t)$ be an
$\bbr^d$-valued 
stationary process satisfying  ${\bf (RV_\alpha)}$  
for some $\alpha>0$.
Then, for any $k\ge 1$, 
\beao
b_{k+1}(\theta)-b_k(\theta)=\E\Big[\Big(\theta'\sum_{t=0}^k\Theta_t\Big)_+^\alpha -\Big(\theta'\sum_{t=1}^k\Theta_t\Big)_+^\alpha \Big].\label{eq:?}
\eeao
\ele
\begin{proof}
We start by observing that each $b_k(\theta)$ can be expressed
in terms of the spectral tail process $(\Theta_t)$. Indeed,
${\bf (RV_\alpha)}$ 
yields for every $k\ge 1$ and
$\theta\in\bbs^{d-1}$ that
\beao
b_k(\theta)&=&\lim_{x\to\infty}\frac{\P(\theta'S_k>x)}{\P(|X|>x)}\\&=&\lim_{\xto}\frac{\P(\cup_{j=1}^k\{\theta'S_k>x,\theta'X_j>x/k\}\cap\{\theta'X_i<x/k,1\le i<j\})}{\P(|X|>x)}\\
&=&
\lim_{\xto}\sum_{j=1}^k \Big[\frac{\P( \theta'S_k>x,\theta'X_j>x/k)}{\P(|X|>x)}-\frac{\P( \theta'S_k>x,\theta'X_j>x/k,\max_{1\le i<j}\theta'X_i>x/k)}{\P(|X|>x)}\Big].
\eeao
By stationarity, the summands in the above expression can be written in
the form
\beao
\frac{\P(|X_0|>x/k)}{\P(|X_0|>x)}\Big[\P\Big(\theta'\sum_{t=1-j}^{k-j}X_t>x,\theta'X_0>x/k \mid
|X_0|>x/k\Big)&&\\
&&\hspace{-6cm}-\P\Big(\theta'\sum_{t=1-j}^{k-j}X_t>x,\theta'X_0>x/k,\max_{1-j\le i<0}\theta'X_i>x/k \mid
|X_0|>x/k\Big)\Big].
\eeao
Here we used the fact that $\{\theta'X_0>x/k\}\subset \{|X_0|>x/k\}$.
Letting $\xto$ in the above expressions, applying the conditional
limits \eqref{eq:sepctraltailprocess} and observing that
$\P(|Y_0|>y)=y^{-\alpha}$, $y>1$,
we obtain the limiting expressions
\beao
\lefteqn{k^\alpha\Big[\P\Big(|Y_0|\theta'\sum_{t=1-j}^{k-j}\Theta_t>k,
|Y_0|\theta'\Theta_{0}>1 \Big)}\\&&-\P\Big(|Y_0|\theta'\sum_{t=1-j}^{k-j}\Theta_t>1,|Y_0|\theta'\Theta_0>1,|Y_0|\max_{1-j\le i<0}\theta'\Theta_i>1\Big)\Big]\\
&=&\E\Big[\Big(\theta'\sum_{t=1-j}^{k-j}\Theta_t\Big)_+^\alpha\wedge
 (k\theta'\Theta_{0})_+^\alpha\Big]-\E\Big[\Big(\theta'\sum_{t=1-j}^{k-j}\Theta_t\Big)_+^\alpha\wedge
 (k\theta'\Theta_{0})_+^\alpha\wedge\max_{1-j\le i<0}(k\theta'\Theta_i )_+^\alpha\Big].
\eeao
Hence $b_k(\theta)$ has \rep
\beao
b_k(\theta)&=&\sum_{j=1}^k \E\Big[\Big(\Big(\theta'\sum_{t=1-j}^{k-j}\Theta_t\Big)_+^\alpha-\max_{1-j\le i<0}(k\theta'\Theta_i )_+^\alpha\Big)_+\wedge
\Big((k\theta'\Theta_{0})_+^\alpha-\max_{1-j\le i<0}(k\theta'\Theta_i )_+^\alpha\Big)_+\Big],
\eeao
and therefore
\beao
\lefteqn{b_{k+1}(\theta)-b_k(\theta)}\\
&=&\E\Big[\Big(\theta'\sum_{t=0}^{k}\Theta_t\Big)_+^\alpha\wedge
 (k\theta'\Theta_{0})_+^\alpha\Big]\\
&&+\sum_{j=1}^k \E\Big[\Big(\Big(\theta'\sum_{t=-j}^{k-j}\Theta_t\Big)_+^\alpha-\max_{-j\le i<0}(k\theta'\Theta_i )_+^\alpha\Big)_+\wedge
\Big((k\theta'\Theta_{0})_+^\alpha-\max_{-j\le i<0}(k\theta'\Theta_i )_+^\alpha\Big)_+\\
&&\hspace{1.2cm}-\Big(\Big(\theta'\sum_{t=1-j}^{k-j}\Theta_t\Big)_+^\alpha-\max_{1-j\le i<0}(k\theta'\Theta_i )_+^\alpha\Big)_+\wedge
\Big((k\theta'\Theta_{0})_+^\alpha-\max_{1-j\le i<0}(k\theta'\Theta_i )_+^\alpha\Big)_+\Big].
\eeao
The expectations in the sum are of the type
$
\E f(\Theta_{-s},\ldots,\Theta_t)$
for integrable $f$ \st\ \\$f(x_{-s},\ldots,x_t)=0$ if
$x_{-s}=0$, $s,t\ge 0$. Then, according to
Theorem 3.1 (iii) in \cite{basrak:segers:2009},
\beao
\E f(\Theta_{-s},\ldots,\Theta_t)= E\Big( f(\Theta_0/|\Theta_s|,\ldots,
\Theta_{t+s}/|\Theta_s|)\,|\Theta_s|^\alpha
\Big) \,,\quad s,t\ge 0\,.
\eeao
Application of this formula and the fact that our \fct s $f$
are homogeneous of order $\alpha$
yield
\beao
b_{k+1}(\theta)-b_k(\theta)
&=&\E\Big[\Big(\theta'\sum_{t=0}^{k}\Theta_t\Big)_+^\alpha\wedge
 (k\theta'\Theta_{0})_+^\alpha\Big]\nonumber\\
&&+\sum_{j=1}^k \E\Big[\Big(\Big(\theta'\sum_{t=0}^{k}\Theta_t\Big)_+^\alpha-\max_{0\le i<j}(k\theta'\Theta_i )_+^\alpha\Big)_+\wedge
\Big((k\theta'\Theta_{j})_+^\alpha-\max_{0\le i<j}(k\theta'\Theta_i
)_+^\alpha\Big)_+\nonumber\\
\eeao
\beao
&&\hspace{1.2cm}-\Big(\Big(\theta'\sum_{t=1}^{k}\Theta_t\Big)_+^\alpha-\max_{1\le i<j}(k\theta'\Theta_i )_+^\alpha\Big)_+\wedge
\Big((k\theta'\Theta_{j})_+^\alpha-\max_{1\le i<j}(k\theta'\Theta_i
)_+^\alpha\Big)_+\Big]\nonumber\\
&=& \E\Big[\sum_{j=0}^k\Big(\Big(\theta'\sum_{t=0}^{k}\Theta_t\Big)_+^\alpha-\max_{0\le i<j}(k\theta'\Theta_i )_+^\alpha\Big)_+\wedge
\Big((k\theta'\Theta_{j})_+^\alpha-\max_{0\le i<j}(k\theta'\Theta_i )_+^\alpha\Big)_+\nonumber\\
&&-\sum_{j=1}^k\Big(\Big(\theta'\sum_{t=1}^{k}\Theta_t\Big)_+^\alpha-\max_{1\le i<j}(k\theta'\Theta_i )_+^\alpha\Big)_+\wedge
\Big((k\theta'\Theta_{j})_+^\alpha-\max_{1\le i<j}(k\theta'\Theta_i
)_+^\alpha\Big)_+\Big]\nonumber\\
&=&\E\Big[\Big(\theta'\sum_{t=0}^k\Theta_t\Big)_+^\alpha -\Big(\theta'\sum_{t=1}^k\Theta_t\Big)_+^\alpha \Big].\nonumber
\eeao
The last identity follows because there exists $\ell=\min\{1\le j\le n;\,(k\theta'\Theta_{j})_+^\alpha \ge \Big(\theta'\sum_{t=1}^{k}\Theta_t\Big)_+^\alpha\}$ such that 
\beao&&\Big(\Big(\theta'\sum_{t=1}^{k}\Theta_t\Big)_+^\alpha-\max_{1\le
  i<j}(k\theta'\Theta_i )_+^\alpha\Big)_+=0\quad\mbox{for all
  $j>\ell,$}
\eeao
and then also 
\beao
&&\Big(\Big(\theta'\sum_{t=1}^{k}\Theta_t\Big)_+^\alpha-\max_{1\le i<\ell}(k\theta'\Theta_i )_+^\alpha\Big)_+\wedge
\Big((k\theta'\Theta_{\ell})_+^\alpha-\max_{1\le i<\ell}(k\theta'\Theta_i )_+^\alpha\Big)_+\\&=&
\Big(\theta'\sum_{t=1}^{k}\Theta_t\Big)_+^\alpha-\max_{1\le
  i<\ell}(k\theta'\Theta_i )_+^\alpha\,,
\eeao
and
\beao
&&\sum_{j=1}^{\ell-1}\Big(\Big(\theta'\sum_{t=1}^{k}\Theta_t\Big)_+^\alpha-\max_{1\le i<j}(k\theta'\Theta_i )_+^\alpha\Big)_+\wedge
\Big((k\theta'\Theta_{j})_+^\alpha-\max_{1\le i<j}(k\theta'\Theta_i )_+^\alpha\Big)_+\\
&=&\sum_{j=1}^{\ell-1}\Big((k\theta'\Theta_{j})_+^\alpha-\max_{1\le i<j}(k\theta'\Theta_i )_+^\alpha\Big)_+\\&=&\sum_{j=1}^{\ell-1} \max_{1\le i\le j}(k\theta'\Theta_i )_+^\alpha-\max_{1\le i<j}(k\theta'\Theta_i )_+^\alpha\\ &=&\max_{1\le i<\ell}(k\theta'\Theta_i )_+^\alpha.
\eeao
\end{proof}
The remainder of this paper crucially depends on the notion of 
{\em cluster index} of the \regvary\ \seq\ $(X_t)$, given as the
limiting \fct :
\beao
b(\theta)=\lim_{\kto} (b_{k+1}(\theta)-b_k(\theta))\,, \quad
\theta\in\bbs^{d-1}\,.
\eeao
  In contrast to the quantities $b_k(\theta)$ the
existence of the limits $b(\theta)$ is not straightforward. 
The following result  yields a sufficient condition for the existence
of $b$.
\bth\label{lem:sum}
Assume that $(X_t)$ satisfies ${\bf (RV_\alpha)}$  
for some $\alpha >0$   
and that $X_t=f(\Phi_t)$, $t\in\bbz$, where $f$ is an
$\bbr^d$-valued \fct\ acting on the \MC\ $(\Phi_t)$ satisfying 
${\bf (DC_{\it p})}$ for some positive $p\in (\alpha-1,\alpha)$.
Then the limits
\beao
b(\theta)&=&\E\Big[\Big(\sum_{t\ge 0}\theta'\Theta_t\Big)^\alpha_+-\Big(\sum_{t\ge 1}\theta'\Theta_t\Big)^\alpha_+\Big]
\,,\quad
\theta\in \bbs^{d-1}\,,
\eeao
exist and are finite. \ethe
\bre 
The cluster index $b$ of $(X_t)$ is a continuous
\fct\ on $\bbs^{d-1}$. This is shown in the proof below: $b$ is the
uniform limit of continuous \fct s on $\bbs^{d-1}$. The index 
$b(\theta)$ is non-negative since it coincides with
the C\`esaro mean $\lim_{\kto} k^{-1} b_k(\theta)$. 
For $0<\alpha\le 1$, the sub-additivity of the function $x\to
x_+^\alpha$ implies the inequality $b(\theta)\le
\E[(\theta'\Theta_0)_+^\alpha]$. Moreover, if
$\E[(\theta'\Theta_0)_+^\alpha]> 0$ then $b(\theta)>0$ by an
application of the mean value theorem when $0<\alpha\le 1$. These two
properties are shared 
by the  extremal index of a multivariate stationary process. The
extremal index  admits a similar representation in terms of 
the spectral tail process, i.e. 
$\E [ (\sup_{t\ge 0}\theta'\Theta_t )^\alpha_+- (\sup_{t\ge 1}\theta'\Theta_t )^\alpha_+ ]$; see \cite{basrak:krizmanic:segers:2011}.\ere
\bre The limit $b$ 
also exists for various classes of \regvary\ stationary processes 
beyond \fct s of a \MC ; see~\cite{bartkiewicz:jakubowski:mikosch:wintenberger:2011,mikosch:wintenberger:2011}
for such examples in
the case $d=1$. The cluster index $b$ plays a crucial role for
characterizing weak and \ld\ limits for partial sums of the processes
$(X_t)$. This was recognized in
\cite{bartkiewicz:jakubowski:mikosch:wintenberger:2011,mikosch:wintenberger:2011},
and we extend some of these results to the multivariate case in Section~\ref{sec:stab}.
\ere
\begin{proof}
We will show that the limit $b(\theta)$ of \eqref{eq:?} exists as
$\kto$.
We start with the case $\alpha>1$. Then, for $x,y\in \bbr$,
 by the mean value theorem,
$|(x+y)^\alpha_+-
x_+^\alpha|\le ( \alpha |y| |x+\xi y|^{\alpha-1})\vee |y|^\alpha$ for some
$\xi\in (0,1)$.
Hence, since $|\theta'\Theta_0|\le 1$ a.s.,
\beao
|b_{k+1}(\theta)-b_k(\theta)|&\le& \E\Big[\Big(\alpha 
\,|\theta'\Theta_0|\,\Big|\sum_{t=1}^k \theta'\Theta_t+\xi
\theta'\Theta_0 \Big|^{\alpha-1}\Big)\vee |\theta'\Theta_0|^\alpha\Big]\\
&\le& \E\Big[\Big(\alpha 
\,\Big|\sum_{t=1}^k \theta'\Theta_t+\xi
\theta'\Theta_0 \Big|^{\alpha-1}\Big)\vee 1\Big]=I_0\,.
\eeao
For $\alpha\in (1,2]$, 
\beao
I_0\le 1+\alpha \sum_{t=0}^k \E|\theta'\Theta_t|^{\alpha-1}\,.
\eeao
We will show that the \rhs\ is finite,
implying that 
$\E\big |\sum_{t=0}^\infty |\theta'\Theta_t|\big|^{\alpha-1}<\infty$ and
$\sum_{t=0}^\infty \theta'\Theta_t$
converges absolutely a.s. An application of Lebesgue dominated
\con\ shows that the limit $b(\theta)$ exists and is finite.
For $\alpha>2$, an application of Minkowski's inequality yields
\beao
I_0\le 1+ \alpha \Big( \sum_{t=0}^k (\E  |\theta'\Theta_t|^{\alpha-1})^{1/(\alpha-1)}\Big)^{\alpha-1}\,.
\eeao
We will show that the \rhs\ is finite and then the same argument as for $\alpha\in (1,2]$ applies. We will achieve the bounds for $I_0$ 
by showing that there exists $c>0$ \st\
\beam\label{eq:mom}
\E  |\theta'\Theta_t|^{\alpha-1}\le c\,\beta^t\,,\quad  t\ge 0\,.
\eeam
Using the fact that $\Theta_t=Y_t/Y_0$ and $Y_0$ are independent, for $t\ge 1$
and $s=\alpha-1$,
\beao
\E|Y_0|^s\,\E|\theta '\Theta_t|^s\le\E|Y_0|^s\, \E|\Theta_t|^s=\E |Y_t|^s\,.
\eeao
By definition of the tail process  and Markov's inequality, for small
$\epsilon>0$ \st\ $s (1+\epsilon)<\alpha$,
\beao
\E |Y_t|^s &=&\int_0^\infty \P(|Y_t|^s>y) \,dy\\
& =&
\int_0^\infty \lim_{\xto} \P(|x^{-1}X_t|^s>y\mid |X_0|>x) \,dy\\
\eeao
\beao
&\le &\int_1^\infty y^{-(1+\epsilon)}\,dy
\lim_{\xto} \dfrac{\E[|X_t|^{s(1+\epsilon)}\,\1_{\{|X_0|>x\}}]}
{x^{s(1+\epsilon)}\P(|X_0|>x)}\\
&&+\int_0^1 y^{-(1-\epsilon)}\,dy
\lim_{\xto} \dfrac{\E[|X_t|^{s(1-\epsilon)}\,\1_{\{|X_0|>x\}}]}
{x^{s(1-\epsilon)}\P(|X_0|>x)}\\
&=& R_1+R_2\,.
\eeao
By virtue of  ${\bf (DC_{\it p})}$ for some 
$p\in (\alpha-1,\alpha)$, using a recursive argument, 
we obtain for sufficiently 
large $y$ and $s(1+\epsilon)\le p$, $$
\E[|X_t|^{s(1+\epsilon)}\mid \Phi_0=y]\le \beta^t |f(y)|^{s(1+\epsilon)} + b\,\sum_{j=1}^t\beta^j.
$$
Using this inequality and Karamata's theorem (see
\cite{bingham:goldie:teugels:1987}), for some $c>0$,
\beao
R_1&\le & c\,\lim_{\xto} \dfrac{\E\big[ \1_{\{|X_0|>x\}}
\E[|X_t|^{s(1+\epsilon)}\mid \Phi_0]\big]}
{x^{s(1+\epsilon)}\P(|X_0|>x)}\\
&\le &c\,\beta^t\,\lim_{\xto} \dfrac{\E[|X_0|^{s(1+\epsilon)} \1_{\{|X_0|>x\}}
]}
{x^{s(1+\epsilon)}\P(|X_0|>x)}
\le c \beta^t\,.
\eeao
Similarly, $R_2\le c \beta^t$. We conclude that \eqref{eq:mom} holds for $\alpha>1$.
\par
It remains to consider the case $\alpha\le 1$. 
We observe that $|(x+y)_+^\alpha -x_+^\alpha|\le |y|^\alpha$ 
for any $x,y\in\bbr$. Hence
\beao
|b_{k+1}(\theta)-b_k(\theta)|\le \E|\theta'\Theta_0|^\alpha\le 1.
\eeao
It suffices to show that $\sum_{t=0}^\infty |\theta'\Theta_t|<\infty$ a.s.
This follows if $\sum_{t=0}^\infty \E |\theta'\Theta_t\big|^s<\infty$ for
some $ s<p$. The proof is analogous, using
${\bf (DC_{\it p})}$ for some $p<\alpha$.
\end{proof}

\section{Limit theory for functions of regenerative \MC s}\label{sec:stab}\setcounter{equation}{0}
In this section we present the main results of this paper.
Throughout we consider an $\bbr^d$-valued process 
$X_t=f(\Phi_t)$, $t\in\bbz$, where $(\Phi_t)$ is an irreducible aperiodic
\MC. We present two types of limit results for the
partial sums $(S_n)$ of $(X_n)$: central limit theory with infinite
stable limits in Theorem~\ref{main} and precise \ld\ results in
Theorem~\ref {thm:ldp}. The proofs of these results are postponed to
Sections~\ref{sec:mainproof} and \ref{sec:proofldp}.

\subsection{Stable limit theory}\label{sec:main}
\setcounter{equation}{0}
We start with a  \clt\ with stable limit law.
\bth\label{main}
Consider an $\bbr^d$-valued strictly stationary \seq\ $(X_t)=(f(\Phi_t))$ satisfying
the following conditions:
\begin{itemize}
\item
${\bf (RV_\alpha)}$ for some $\alpha\in (0,2)$, $\E X=0$ if $\alpha>1$ and $X$ is symmetric if $\alpha=1$.
\item
 ${\bf (DC_{\it p})}$ for some $p\in ((\alpha-1)\vee 0,\alpha)$.
\end{itemize}
Let $(a_n)$ be a \seq\ of positive numbers \st\
$n\,\P(|X_0|>a_n)\sim 1$. Then the following statements hold:
\begin{enumerate}
\item
The \clt\ 
$
a_n^{-1} S_n \std \xi_\alpha
$ is satisfied for a centered $\alpha$-stable random vector
$\xi_\alpha$  with spectral \ms\ $\Gamma_\alpha$ 
on $\bbs^{d-1}$  (see 
\cite{samorodnitsky:taqqu:1994}, Section 2.3, for a definition)
given by the relation
\beam\label{eq:b1}
b(\theta)=C_\alpha\,\int_{\bbs^{d-1}}
(\theta's)^\alpha_+\Gamma_\alpha(ds)
\,,\quad \theta\in\bbs^{d-1}\,,
\eeam
where $b$ is the cluster index of $(X_t)$ introduced in Section~\ref{subsec:cluster}
and
\beam\label{eq:calpha}
C_\alpha&=& \dfrac{1-\alpha}{\Gamma(2-\alpha) \cos(\pi\alpha/2)}\,.
\eeam
If $b\equiv 0$ the limit $\xi_\alpha=0$ a.s.
\item
If $b\ne 0$ the partial sums over full cycles $(S(i))_{i=1,2,\ldots}$  defined in \eqref{eq:partial}
are \regvary\ with index $\alpha$ and spectral \ms\ $\P_{\Theta'}(\cdot)$
on $\bbs^{d-1}$ given by
\beam\label{eq:seopc}
d\P_{\Theta'}(ds)
=\dfrac{b(s) }{\int_{\bbs^{d-1}}b(\theta)\,d\P_{\Theta}(\theta)}\, d\P_{\Theta}(ds)\,.
\eeam
\end{enumerate}
\ethe
The proof of Theorem~\ref{main} is given in
Section~\ref{sec:mainproof}.
 To a large extent, the results of
Theorem~\ref{main} can be extended to the case of non-irreducible \MC
s. A short discussion of this topic will be given at the end of Section~\ref{sec:mainproof}.

\subsection*{ A discussion of related stable limit
  results}\label{subsec:rel}
 Theorem~\ref{main} complements the \clt\ with Gaussian limits 
for $\bbr^d$-valued \fct s of a \MC ; see Theorem~\ref{thm:samur}
above. For both results, conditions of type  
 ${\bf (DC_{\it p})}$  enter the proofs to show the existence of
moments of $S(1)$ under the existence of the corresponding moments for
$X_0$.
\par
 The history of stable limit theory for non-linear multivariate 
\ts\ is short in comparison with the finite variance case.
Davis and Mikosch \cite{davis:mikosch:1998} prove a \clt\ with
$\alpha$-stable limit for an $\bbr^d$-valued strictly stationary \seq\
$(X_t)$, satisfying a weak dependence condition.  
The result is a straightforward extension of the 1-dimensional
result proved in Theorem 3.1 of Davis and Hsing \cite{davis:hsing:1995}.
We recall the forementioned results for the reason of comparison 
with Theorem~\ref{main}.
\bth\label{thm:davhs}
Assume that the strictly stationary $\bbr^d$-valued 
\seq\ $(X_t)$ satisfies ${\bf (RV_\alpha)}$ for some $\alpha>0$ 
and the following \pp\ \con\ result holds:
\beao
N_n=\sum_{t=1}^n\delta_{a_n^{-1} X_t}\std N=
\sum_{i=1}^\infty\sum_{j=1}^\infty\delta_{P_i Q_{ij}}\,,
\eeao 
where $(P_i)$ are the points of a 
Poisson random \ms\ on $(0,\infty)$ with intensity 
$h(y)=\gamma \alpha y^{-\alpha-1}$, $y>0$, and it is assumed that 
$\gamma>0$,\footnote{  Basrak and Segers \cite{basrak:segers:2009},
Proposition~4.2, show that $\gamma>0$ is automatic if $(X_t)$ satisfies their
anti-clustering Condition 4.1 and a modification of the mixing 
${\mathcal A}(a_n)$ from \cite{davis:hsing:1995}. Both conditions are
very mild. The quantity $\gamma$ is known as the {\em extremal index}
of the \seq\  $(X_t)$; see \cite{leadbetter:lindgren:rootzen:1983}.}
the \seq\ $(Q_{ij})_{j\ge 1}$, $i=1,2,\ldots,$
is iid with values $|Q_{ij}|\le 1$, independent of $(P_i)$ 
and \st\ $sup_{j\ge 1}|Q_{ij}|=1$.
\begin{enumerate}
\item[\rm (1)] If $\alpha\in (0,1)$ then 
\beao
a_n^{-1} S_n\std \xi_\alpha=\sum_{i=1}^\infty\sum_{j=1}^\infty P_i Q_{ij}
\eeao
and $\xi_\alpha$ has an $\alpha$-stable \ds ,
\item[\rm (2)]
If $\alpha\in [1,2)$ and for any $\delta>0$,
\beam\label{eq:sv}
\lim_{\vep\downarrow 0}\limsup_{\nto}\P (|S_n(0,\vep]- 
\E S_n(0,\vep] |>\delta)=0\,,
\eeam
where $S_n(0,\vep]=a_n^{-1}\sum_{t=1}^n X_t \1_{\{|X_t|\le \vep a_n\}}$, then 
\beao
a_n^{-1} S_n- \E S_n(0,1]\std \xi_\alpha\,,
\eeao
where $\xi_\alpha$ is the \ds al limit as $\vep\downarrow 0$ of 
\beao
\Big(\sum_{i=1}^\infty\sum_{j=1}^\infty P_i Q_{ij} \1_{(\vep,\infty)}( P_i |Q_{ij}|)
-\int_{\vep<|x|\le 1} x\,\mu_X(dx)\Big)
\eeao
which exists and has an $\alpha$-stable \ds . (Recall that $\mu_X$ is
the limit \ms\ in \eqref{eq:1}.)
\end{enumerate}
\ethe
The latter result has been the basis for a variety of results
for partial sums of strictly stationary processes 
with infinite variance stable limits; 
see 
\cite{davis:mikosch:1998,mikosch:starica:2000,basrak:krizmanic:segers:2011,tyran:2010}. The main idea of the proof of  Theorem~\ref{thm:davhs}
is a continuous mapping argument acting on $N_n\std N$, showing
that the sums of the points of $N_n$ converge in \ds\ to the
corresponding sum of the points of $N$. This method is rather
elegant and can be applied to a large variety of strictly stationary 
\regvary\ vector \seq s $(X_t)$. The proofs use advanced point process
techniques.
\par
A characterization of the parameters of the \ds\ of the 
multivariate limit $\xi_\alpha$ in Theorem~\ref{thm:davhs}
can be given by extending 
Theorem 3.2 in \cite{davis:hsing:1995} to the multidimensional case:
if 
\beam\label{eq:36}
\E(\sum_{j\ge 1} |Q_{1j}|)^\alpha<\infty
\eeam 
then the \levy\ spectral
\ms\ $\Gamma_\alpha$ of $\xi_\alpha$ is described by
\beao
\int_{\bbs^{d-1}} (\theta's)^\alpha_+\Gamma_\alpha(ds)=\gamma \frac{\alpha}{2-\alpha}\,\E\Big[\Big(\sum_{t\ge 1}\theta'Q_{1t}\Big)^\alpha_+ \Big]\,,\quad \theta\in\bbs^{d-1}\,.
\eeao
This representation is particularly useful for $\alpha<1$. Then
\eqref{eq:36} is always satisfied. 
Adapting Theorem \ref{thm:davhs} in terms of the tail process 
as in Basrak et al. \cite{basrak:krizmanic:segers:2011}, 
an alternative characterization of the \levy\ spectral measure 
$\Gamma_\alpha$ is the following:
if 
\beam\label{eq:37}
\E(\sum_{t\ge 0} |\Theta_{t}|)^\alpha<\infty\eeam
 then 
$$
\int_{\bbs^{d-1}} (\theta's)^\alpha_+\Gamma_\alpha(ds)=  C_\alpha^{-1}\E\Big[\Big(\sum_{t\ge 0}\theta'\Theta_{t}\Big)^\alpha_+\1_{\{\Theta_i=0,\,\forall i\le -1\}} \Big]\,,\quad \theta\in\bbs^{d-1}\,.
$$
Conditions \eqref{eq:36} and \eqref{eq:37} may fail for $\alpha>1$,
e.g. for a GARCH(1,1) model; see Section~\ref{exam:garch}. 
\par
 If we assume the conditions of Theorem~\ref{main}, classical
computation for $\alpha\ne 1$ yields
\beao
\E[\exp(iv'\xi_\alpha)]&=&\exp(-\int_{\bbs^{d-1}}|v'\theta|^\alpha(1-i\mbox{sign}(v'\theta)\tan(\pi\alpha/2))\Gamma_\alpha(d\theta))\\
&=&\exp\Big(\int_0^\infty\E\Big[\exp\Big(iu\sum_{t=1}^\infty v'\Theta_t\Big)-\exp\Big(iu\sum_{t=0}^\infty v'\Theta_t\Big)\Big]\alpha x^{-\alpha-1}dx\Big)\,.
\eeao
For $\alpha\in (0,1)$, this form of the limiting stable \chf\ was
proved in Basrak and Segers \cite{basrak:segers:2009}. 
\par
The additional condition \eqref{eq:sv}
is not easily checked for dependent \seq s. 
It is implied for stationary $\rho$-mixing processes with rate \fct\ 
$\rho(j)$ satisfying $\sum_{j\ge 1}\rho(2^j)<\infty$; 
 see \cite{jakubowski:1997}. 
It is also implied by  ${\bf (DC_{\it p})}$ for functions of an 
irreducible \MC ;  see \cite{mikosch:wintenberger:2011}. 
For  a  (possibly non-irreducible) \MC\ $(X_t)$, 
condition   ${\bf (DC_{\it p})}$ is much weaker than this $\rho$-mixing condition
which  is equivalent to a spectral gap in $L^2(\P)$; see
\cite{meyn:tweedie:1993}.  
\par
In our paper, \chf\ based 
methods are employed which are close to those used in classical limit
theory for iid \seq s; see e.g. \cite{petrov:1995}.
As in the iid case, Theorem \ref{main} yields an explicit form
of the \chf\ of the limiting $\alpha$-stable random vector.
The underlying extremal dependence structure of $(X_t)$ shows
via the cluster  index $b(\theta)$ which appears
explicitly in the \chf .   We refer the reader to the
extensive discussion in 
\cite{bartkiewicz:jakubowski:mikosch:wintenberger:2011} on 
the comparison of the point process and 
the \chf\ approaches to stable limit theory. One
drawback of our approach is that, in contrast to
the point process approach, 
we do not have series representations of $\xi_\alpha$ in terms 
of the \seq\ $(\Theta_t)$.
\par
Recently, the special case of solutions to   
multivariate \sre s \eqref{eq:sre} has attracted attention; see e.g. 
\cite{damek:mentemeier:mirek:zienkiewicz:2011,buraczewski:damek:guivarc'h:2010}. In this case, one can exploit the underlying  random 
iterative contractive structure to derive stable limits without 
additional restrictions. We mention that drift  conditions 
such as ${\bf (DC_{\it p})}$ are automatically satisfied for 
solutions of \sre s; see Section \ref{sec:ex}.

\subsection{Precise large deviations for functions of a \MC}\label{sec:ldp}\setcounter{equation}{0}
In this section, we extend some of the results obtained in 
\cite{mikosch:wintenberger:2011} for general univariate \regvary\ 
\seq s.\footnote{ For comparison and since we will use it in the proofs, 
we quote the main result of 
\cite{mikosch:wintenberger:2011} as Theorem~\ref{ldMW}.} 
 We again focus on $\bbr^d$-valued \seq s  
$(X_t)=(f(\Phi_t))$ for an underlying aperiodic irreducible \MC\ $(\Phi_t)$. 
The case $\alpha\in (0,2)$ turns out to be a con\seq\ of
Theorem~\ref{main}; the proof is given in Section~\ref{sec:smallalpha}.
The proof in the case $\alpha>2$ is more involved and requires different
techniques; see Section~\ref{sec:largealpha}. 
\bth\label{thm:ldp}
{ Consider an $\bbr^d$-valued
strictly stationary \seq\ $(X_t)=(f(\Phi_t))$ for an aperiodic irreducible 
\MC\ $(\Phi_t)$. Assume that $(X_t)$  satisfies the condition 
${\bf (RV_\alpha)}$ for some $\alpha>0$. Let 
$(\la_n)$ be any \seq\ such that 
$\log(\la_n)=o(n)$ and
$\la_n/n^{1/\alpha+\vep}\to
\infty$ if $\alpha\in (0,2)$ and 
$\la_n/n^{0.5+\vep}\to\infty$ if $\alpha>2$ 
for any 
$\vep>0$. Assume either 
\begin{enumerate}
\item[(1)]
$\alpha\in (0,2)$ and the conditions of 
Theorem~\ref{main} are satisfied, or
\item[(2)]
$\alpha>2$, $\alpha\not\in \bbn$ or
$b(\theta)=b(-\theta)$, $\theta\in \bbs^{d-1}$, and ${\bf (DC_{\it p})}$ holds for every $p<\alpha$,
\end{enumerate}
then the following \ld\ principle holds:
\beam\label{eq:lda}
\dfrac{\P (\lambda_n^{-1}S_n\in\cdot)}{n\,\P(|X|>\lambda_n)}\stv
\nu_\alpha\,,\quad \nto\,,
\eeam
where $\nu_\alpha$ is a Radon \ms\ on the Borel $\sigma$-field of $\ov
\bbr^d_0$  
uniquely determined by the relations
\beam\label{eq:cb}
\nu_\alpha(t\{x: \theta'x >1\})= t^{-\alpha} \nu_\alpha(\{x: \theta'x >1\})=t^{-\alpha} \,b(\theta)\,,\quad \theta\in\bbs^{d-1},\; t>0\,.
\eeam}
\ethe
\bre 
The conditions $\alpha\not\in \bbn$ or $b(\cdot)=b(-\cdot)$
are needed to apply inverse results for \regvar .
For $\alpha>2$,  
we show that the \ms\ $\nu_\alpha$ on the Borel $\sigma$-field of 
$\ov \bbr_0^d$ is uniquely determined by its values on sets of the form
$t\{x: \theta'x >1\}$, $t>0$, $\theta\in\bbs^{d-1}$, provided  the
mentioned additional conditions are met. In general, such conditions
cannot be avoided; \cite{kesten:1973,hult:lindskog:2006a} give counterexamples 
for integer values $\alpha$.
In \cite{basrak:davis:mikosch:2002,boman:lindskog:2009,kluppelberg:pergam:2007}
further conditions on the vector $X$ are given which allow one to 
discover the \ms\ $\nu_\alpha$ from its knowledge on the sets
$t\{x: \theta'x >1\}$, $t>0$, $\theta\in\bbs^{d-1}$.
\ere
\bre
The proof of Theorem~\ref{thm:ldp} shows that \eqref{eq:lda} holds 
uniformly for certain intervals of normalizations and 
for half-spaces not containing the origin. To be precise, the following 
uniform relations hold
 \beam\label{eq:rig}
\lim_{n\to \infty}\sup_{x\in \Lambda_n}\Big|\frac{\P(\theta'S_n> x)}{n\,\P(|X|> x)}-b(\theta) \Big|=0\,,\quad \theta\in \bbs^{d-1}\,,
\eeam
for regions $\Lambda_n=(b_n,c_n)$. Here 
$(b_n)$
satisfies $b_n=n^{0.5+\vep}$ in the case $\alpha>2$ and 
 $b_n=n^{1/\alpha+\vep}$ in the case $\alpha\in (0,2)$ for any $\vep>0$, 
and $(c_n)$ is chosen \st\ $c_n>b_n$ and $\log c_n=o(n)$. 
  Moreover, for \eqref{eq:rig} one does not need the additional
  conditions $b(\cdot)=b(-\cdot)$ and $\alpha\not\in\bbn$.
\ere

\section{Examples}\label{sec:ex}\setcounter{equation}{0}
 Here we consider several examples of \regvary\ stationary processes with
index $\alpha>0$, 
where the theory of the previous sections 
applies. In particular, we will determine the tail process 
$(\Theta_t)$, the cluster index $b$ and verify the drift condition
 ${\bf (DC_{\it p})}$ for $p<\alpha$. All models considered fall in the class
of \fct s acting on an aperiodic irreducible \MC .

\subsection{Vector-autoregressive process}\label{exam:ar1}
Consider the vector-autoregressive process of order 1 given
by
\beam\label{eq:ar1}
X_t = A\,X_{t-1} + Z_t\,,\quad t\in\bbz\,,
\eeam
where  $A$ is a random $d\times d$ matrix 
whose eigenvalues are
less than 1 in absolute value, and $A$ is independent of the 
iid $\bbr^d$-valued 
\seq\ $(Z_t)$  which is \regvary\ with index
$\alpha>0$. Then we also have $\E\|A\|^s<1$ for every $s>0$.
Here $\|\cdot\|$ denotes the operator norm with respect to the Euclidean norm.

Then a stationary solution $(X_t)$  to \eqref{eq:ar1}
exists and has \rep 
\beao
X_t= A^t X_0 + \sum_{i=1}^t A^{t-i} Z_i\,,\quad t\ge 0\,;
\eeao 
see \cite{brockwell:davis:1991}, Chapter 11. 
Morever,  $X_0$ is \regvary\ with index $\alpha$; see \cite{resnick:willekens:1991}. In particular, denoting the limiting \ms\ of the \regvary\ 
vector $Z_0$ by $\mu_Z$, it follows from \cite{resnick:willekens:1991} that
\beam\label{eq:willekens}
\dfrac{\P( x^{-1} X_0\in \cdot)}{\P(|Z_0|>x)}\stv
\sum_{i=0}^\infty \E\big[\mu_{Z}(\{x\in \bbr^d: A^i x\in \cdot\}\big)\big]\,. 
\eeam
Since 
\beao
(X_1,\ldots,X_h)= (A,\ldots, A^h) X_0+ 
\Big(Z_1,\ldots, \sum_{t=1}^h A^{h-t} Z_t
\Big)\,,
\eeao
and $(Z_t)_{t\ge 1}$ is independent of $X_0$, \regvar\ of $(X_1,\ldots,X_h)$
is a con\seq\ of the fact that \regvar\ is kept under linear transformations.
Let $C$ be a continuity set relative to the limiting 
\ms\ $\mu_{h+1}$ of $(X_0,\ldots,X_h)$ and $I_d$ the identity matrix. Since $X_0$ is independent of $(Z_t)_{t\ge 1}$, 
\beao
\P(x^{-1} (X_0,\ldots,X_h)\in C \mid |X_0|>x)
&=&\P(x^{-1}(I_d,A,\ldots, A^h) \,X_0  \in C \mid |X_0|>x)\\
&&+\P(x^{-1}(0,Z_1, \sum_{i=1}^2 A^{2-i} Z_i,\ldots, \sum_{i=1}^h A^{h-i}
Z_i\in C) + o(1)\\
&\to  &\P((I_d,A,\ldots ,A^h) Y_0\in C) \,,\quad \xto\,.
\eeao
Thus we may identify $(\Theta_t)_{t=0,\ldots,h}$ with 
$(I_d,A,\ldots, A^h) \Theta_0$.
In view of \eqref{eq:willekens},
\beao
\P( x^{-1} X_0/|X_0|\in \cdot\mid |X_0|>x)\stw
\dfrac{\sum_{i=0}^\infty \E\big[\mu_{Z}(\{x\in \bbr^d: A^i x/|A^i x|\in
  \cdot\,,|A^i x|>1\}\big)\big]}{\sum_{i=0}^\infty \E\big[\mu_{Z}(\{x\in \bbr^d: |A^i
    x|>1\})\big]}=\P(\Theta_0\in \cdot)\,.
\eeao
Writing $(I_d-A)^{-1}=\sum_{t=0}^\infty A^t$ (this series converges since the largest eigenvalue of $A$ is smaller than 1), we conclude that 
\beao
b(\theta)=
\E\Big[\big(\theta' (I_d-A)^{-1}\Theta_0\big)_+^{\alpha}
-\big(\theta' A (I_d-A)^{-1} \Theta_0\big)_+^{\alpha}
\Big]\,,\quad \theta\in \bbs^{d-1}\,.
\eeao
Next we show  ${\bf (DC_{\it p})}$ for $p<\alpha$. First assume $p>1$.
A Taylor series expansion yields
\beao
\E( |A x+Z_1|^p- |A x|^p) &\le &p \E [|Z_1| \,|Ax + \xi Z_1|^{p-1}]\\
&\le &c\,(\E|Ax|^{p-1}+1)\le c\,( |x|^{p-1} +1)
\eeao
for some \rv\ $\xi\in (0,1)$ a.s.
Then for some $\beta\in (\E\|A\|^p,1)$ and sufficiently large $|x|$,
\beao
\E |A x+Z_1|^p\le \E|A x|^p + c \,(1+ |x|^{p-1})\le 
\E\|A\|^p |x|^p (1+c\,|x|^{-1}) +c\le \beta |x|^p +c\,,
\eeao
and  ${\bf (DC_{\it p})}$ is satisfied. If $p\le 1$ a simpler 
argument applies with $\beta=\E\|A\|$:
\beao
\E( |A x+Z_1|^p\le  \E|A x|^p +\E|Z_1|^p\le \beta \,|x|^p +c\,.
\eeao
  If the \MC\ $(X_t)$ is also aperiodic and irreducible
the results in Section~\ref{sec:stab} are  directly applicable with $f(x)=x$.
\subsection{Random affine mapping}\label{exam:1}
Following Kesten \cite{kesten:1973}, we consider the
\sre
\beam\label{eq:sre}
X_t= A_t\,X_{t-1}+ B_t\,,\quad t \in\bbz\,,
\eeam
where $((A_t,B_t))_{t\in\bbz}$ is an iid \seq , $A_t$ are
random $d\times d$-matrices and $B_t$ are $\bbr^d$-valued random
vectors. We also assume $\E \log ^+\|A\|<\infty$, where $\|\cdot\|$ denotes the operator norm \wrt\ the Euclidean norm,
$\E \log ^+|B|<\infty$, and that the Lyapunov exponent of the \sre\ \eqref{eq:sre} is negative. 
These conditions ensure that
an a.s.  unique stationary causal solution  $(X_t)$ to  \eqref{eq:sre}
exists; see \cite{bougerol:picard:1992}.  
Under additional regularity conditions which ensure that the \ds\ 
of $A$ is sufficiently spread out, 
the equation   
\beam\label{eq:lyap}
\varrho(\kappa)=\lim_{\nto} n^{-1} \log \E\|A_1\cdots
A_n\|^\kappa=0\,,\quad \kappa>0\,,
\eeam
has a  unique positive solution $\alpha$ and $\theta'X$,
$\theta\in\bbs^{d-1}$, is \regvary\ with index $\alpha$.
 Under stronger conditions on $A$, $\alpha$ can be
calculated as the solution to
$E\|A\|^\kappa=1$, $\kappa>0$; see
\cite{gao:guivarch:lepage:2011,buraczewski:damek:guivarc'h:hulanicki:urban:2011}
for recent results. 
Kesten  \cite{kesten:1973} had already given conditions which ensured that at least one of the 
linear combinations $\theta'X$, $\theta\in\bbs^{d-1}$,
is \regvary\ with index $\alpha$. In general, one cannot conclude from \regvar\ of  $\theta'X$, $\theta\in\bbs^{d-1}$,
that $X$ is \regvary ; see \cite{kesten:1973,hult:lindskog:2006a} for some counterexamples.\footnote{However, it might be possible to prove \regvar\ of $X_t$ by using the structure of \eqref{eq:sre}.}
In \cite{basrak:davis:mikosch:2002,boman:lindskog:2009,kluppelberg:pergam:2007}
conditions are given which ensure that the \regvar\
of a vector can be recovered from the \regvar\ of its linear 
projections. One of
these conditions is that $\alpha \not\in \bbn$; 
see \cite{basrak:davis:mikosch:2002} for details.
 In what follows, we will assume that $X_t$ is \regvary\ with 
index $\alpha>0$ and that the stronger moment conditions 
$\E\|A\|^{2(\alpha+\epsilon)}<\infty$ and 
$\E|B|^{2(\alpha+\epsilon)}<\infty$
hold for some $\epsilon>0$. If $A_t$ and $B_t$ are independent the 
milder moment conditions
$\E\|A\|^{\alpha+\epsilon}<\infty$ and
$\E|B|^{\alpha+\epsilon}<\infty$ for some $\epsilon>0$ suffice.
\par
Calculation yields
\beam\label{eq:sre2}
X_t=\Pi_t X_0 +R_t\,,\quad \mbox{where}\quad \Pi_t=A_t\cdots A_1\,,\quad
t\ge 1\,,
\eeam
where $\E |R_t|^{\alpha+\epsilon}<\infty$
and hence
\beao
\P( x^{-1} (X_0,\ldots,X_t)\in \cdot \mid |X_0|>x)&\stw& \P(|Y_0|\,
(I_d,\Pi_1,\ldots \Pi_t)\Theta_0\in\cdot )\,.
\eeao
where $\P (X_0/|X_0|\in \cdot\mid  |X_0|>x)\stw \P(\Theta_0\in
\cdot)$ and $\Theta_0$ is independent of $(A_t)_{t\ge 1}$.
Therefore $(\Theta_i)_{i=0,\ldots,t}=(I_d,\Pi_1,\ldots, \Pi_t)\Theta_0$.
Writing $\Pi_0=I_d$, the identity matrix in $\bbr^d$, 
and $(Z_t)$ for the solution of the \sre\ \eqref{eq:sre}
in the special case $B=I_d$, 
we obtain from Theorem~\ref{lem:sum},
\beam\label{eq:jak}
b(\theta)&=& \E\Big[\Big(\theta'\sum_{t\ge 0} \Pi_t
\Theta_0\Big)^\alpha_+-\Big(\theta'\sum_{t\ge
  1}\Pi_t\Theta_0\Big)^\alpha_+\Big]\nonumber\\
&=&\E\Big[\Big(\theta'(Z_1+I_d)
\Theta_0\Big)^\alpha_+-\Big(\theta'Z_1\Theta_0\Big)^\alpha_+\Big]
\,,\quad
\theta\in \bbs^{d-1}\,,
\eeam
provided we can show  ${\bf (DC_{\it p})}$ 
for the \MC\ $(\Phi_t)=(X_t)$.
The formula \eqref{eq:jak} is in agreement with the calculations for $d=1$ in
\cite{bartkiewicz:jakubowski:mikosch:wintenberger:2011}.
\par 
Since  \eqref{eq:lyap}
is satisfied we can use Lemma \ref{lem:skel} for proving ${\bf (DC_{\it p})}$.
Assuming irreducibility and aperiodicity of the \MC\ $(X_t)$ and
exploiting the definition of $\alpha$ as solution to \eqref{eq:lyap},  
one can choose $m\ge 1$ sufficiently large such that 
$\E\|A_1\cdots
A_m\|^p<1$ for any $p<\alpha$. Indeed, assume on the contrary that 
$\varrho(p)\ge 0$ for some $p<\alpha$. This contradicts the convexity
of $\varrho$ which has roots at $0$ and $\alpha$.
Then the $m$-skeleton of the chain
satisfies the drift condition, ${\bf (DC_{\it p,m})}$ follows and 
Lemma  \ref{lem:skel} yields ${\bf (DC_{\it p})}$.
Thus we conclude that the results of Section~\ref{sec:stab}
are directly applicable to the \MC\ $(X_t)$   with $f(x)=x$  if it is also
aperiodic and irreducible. 
\subsection{Sample autocovariance \fct\ of one-dimensional random affine
  mapping}\label{exam:acf}
Consider the solution $(X_t)$ to the \sre\ \eqref{eq:sre} in
  the case  $d=1$, under irreducibility and aperiodicity. We assume the
conditions and use the notation of Section~\ref{exam:1}. 
In addition, we write
\beao
\Pi_{s,t}=\left\{\barr{ll} A_s\cdots A_t& s\le t\,,\\
1& \mbox{otherwise.}\earr\right.
\eeao
In particular, we assume that $(X_t)$ is \regvary\ with index
$\alpha>0$ satisfying $\E|A|^\alpha=1$,  $\E|A|^{\alpha+\vep}<\infty$ and $\E|B|^{\alpha+\varepsilon}<\infty$ for some $\varepsilon>0$. For $h\ge 0$,
consider the process $\Phi_t=(X_t,X_{t-1},\ldots,X_{t-h})'$,
$t\in\bbz$, of lagged vectors. They 
constitute an $\bbr^{h+1}$-valued stationary, aperiodic and irreducible
\MC. Similar arguments as in
 Section~\ref{exam:1} 
show that the chain is \regvary\ with index $\alpha>0$. 
We consider the   following 
\fct\ acting on the \MC\ $(\Phi_t)$:
\beao
  {\bfX}_t= f(\Phi_t)= (X_t \Phi_t,X_t,\Phi_t)\,,\quad t\in\bbz\,.
\eeao
  By convention, we will assume that all vectors are understood
  as column vectors. 
The \seq\ $(\Phi_t)$ satisfies the recursion
$$
\Phi_t=\left(\begin{matrix}A_t&0& \cdots&0&0\\ 1&0&\cdots&0&0\\0&1&\cdots&0&0\\
\vdots&\vdots&\ddots&\vdots&\vdots \\
0&0&\cdots &1&0
\end{matrix}\right) \Phi_{t-1} +\left(\begin{matrix}B_t\\0\\0\\\vdots\\
  0\end{matrix}\right)= {\bf A}_t  \Phi_{t-1}+ {\bf B}_t\,,\quad t\in\bbz\,.
$$
We will show that $(\bfX_t)$ satisfies 
${\bf (RV_{\alpha/{\rm 2}})}$ and ${\bf (DC_{\it p,m})}$ for $m$
sufficiently large, $V(x)=|x|^p$ and $p<\alpha/2$.  
The condition $ \E(
V(f(\Phi_1))\mid \Phi_0=y)\le c_1V(f(y))+c_2$ for some positive
$c_1,c_2$ follows immediately from
the \sre\ 
$$
{\bfX}_t=\left(\begin{matrix}A_t{\bf A}_t&A_t{\bf B}_t& B_t{\bf A}_t\\ 0_{1,h+1}&A_t&0_{1,h+1}\\0_{h+1,h+1}&0_{h+1,1}&{\bf A}_t
\end{matrix}\right) {\bfX}_{t-1} +\left(\begin{matrix}B_t{\bf B}_t\\B_t\\
 {\bf B}_t\end{matrix}\right)={\bf C}_t {\bfX}_{t-1}+{\bf D}_t\,,\quad t\in\bbz\,.
$$
Condition $\E|{\bf D}|^{(\alpha+\varepsilon)/2}<\infty$ follows by the
assumptions.
From basic algebra,  for $m\ge h$ the matrix 
products $\prod_{t=1}^mA_t{\bf A}_t=\Pi_t \prod_{t=1}^m \bfA_t$ 
can be written as $\Pi_{m,h}{\bf M}_h$, where the $(h+1)\times (h+1)$ 
matrix ${\bf M}_h$ has zero entries but the first column given by
$(\Pi_{1,h-1},\Pi_{2,h-1},\ldots,1)$.
Products of triangular matrices remain triangular  and  
their diagonal is the product of the diagonals. Thus we obtain
$$
{\bf C}_m\cdots {\bf C}_1=\left(\begin{matrix}\Pi_{m,h}^2I_h&0_{1,h+1}& 0_{h+1,h+1}\\ 0_{1,h+1}&\Pi_{m,h}
&0_{1,h+1}\\0_{h+1,h+1}&0_{h+1,1}&\Pi_{m,h} I_h
\end{matrix}\right) \wt {\bf C}_h= \wt {\bf D}_m \wt {\bf C}_h\,,
$$
where $\wt {\bf C}_h$ is an upper triangular block matrix 
depending only on  $(A_t)_{1\le t\le h-1}$.  The matrices $ \wt {\bf D}_m$ and $\wt {\bf C}_h$ are independent and for some $c>0$ we have
\beao
\E\|\wt {\bf D}_m \wt {\bf C}_h\|^p&\le& \E\|\wt {\bf D}_m\|^p\E\| \wt {\bf C}_h\|^p\\
&\le& c\E[|A_m|^{2p}\cdots |A_{h}|^{2p}+|A_m|^{p}\cdots |A_{h}|^{p}]\E\| \wt {\bf C}_h\|^p.
\eeao
Since $p<\alpha/2$, $\E(|A_0|^{2p})^m\to 0$ and  $\E(|A_0|^{p})^m\to
0$ as $m\to \infty$. Thus, for  $m$ sufficiently large, 
$\E\|\wt {\bf D}_m \wt {\bf C}_h\|^p\le c\, (\E(|A_0|^{2p})^m +
\E(|A_0|^{p})^m)<1$, i.e. condition ${\bf (DC_{\it p,m})}$ holds,
 and Lemma \ref{lem:skel} applies provided we can also show ${\bf
  (RV_{\alpha/{\rm 2}})}$ for $(\bfX_t)$. This is our next goal.
Since $X_t$ and $\Phi_t$ are \regvary\ with
index $\alpha$ we deal with a degenerate case where the limiting
\ms\ of \regvar\ of $\bfX_t$ is concentrated at zero for the last
$h+2$ components. Then, in view of the definition of the cluster
index, $b$ is the same for $(X_t\Phi_t)$ and  $(\bfX_t)$.
Therefore we will calculate $b$ for $(X_t\Phi_t)$. Abusing notation,
we will also use the same notation for the tail process.
 As in Section~\ref{exam:1} 
we obtain by iteration of the \sre\ $X_t=A_tX_{t-1}+B_t$,
\beam\label{eq:iter}
X_{t} \Phi_{t}&=&\Pi_{t-h+1,t}
(\Pi_{t-h+1,t},\Pi_{t-h+1,t-1},\ldots,1)' X_{t-h}^2 +\bfR_t^{(1)}\nonumber\\
&=&\Pi_{1-h,t-h}^2\Pi_{t-h+1,t}
(\Pi_{t-h+1,t},\Pi_{t-h+1,t-1},\ldots,1)' X_{-h}^2 +\bfR_t^{(2)}\,,\nonumber\\
&=& \Pi_{1-h ,t}
(\Pi_{1-h,t},\Pi_{1-h,t-1},\ldots,\Pi_{1-h,t-h})' X_{-h}^2 +\bfR_t^{(2)}\,,
\eeam
where $\E|\bfR_t^{(i)}|^{(\alpha+\vep)/2}<\infty$, $i=1,2$. Then for  $t\ge 0$,
\beao\lefteqn{
(X_0\Phi_0,\ldots,X_t \Phi_t)}\\
&=&\left(\begin{matrix}\Pi_{1-h,0}^2 &\Pi_{1-h,1}^2
&\cdots& \Pi_{1-h,t}^2
\\\Pi_{1-h,0} \Pi_{1-h,-1}  &\Pi_{1-h,1}\Pi_{1-h,0} 
&\cdots&\Pi_{1-h,t} \Pi_{1-h,t-1}\\\vdots&\vdots&\ddots&\vdots\\\Pi_{1-h,0}
& \Pi_{1-h,1}A_{1-h}&\cdots &\Pi_{1-h,t}\Pi_{1-h,t-h}
\end{matrix}\right) X_{-h}^2
+\bfQ_t\,,\quad t\in\bbz\,.
\eeao
and $\E|\bfQ_t|^{(\alpha+\vep)/2}<\infty$. In the remainder of this
section we assume that $P(A=0)=0$; the general case can be treated as well
but leads to tedious case studies.
An application of Corollary 3.2 in Basrak and Segers
\cite{basrak:segers:2009} yields that for continuity sets $M$,
\beao
\P(x^{-1} (X_0\Phi_0,\ldots,X_t \Phi_t) \in M \mid |X_0\Phi_0|>x)
\to P(|Y_0| \bfE_t \in M)\,,
\eeao
where 
\beao
\bfE_t \eqd \frac{1}{|\Pi_h|\sqrt{\Pi_{h}^2+\Pi_{h-1}^2+\cdots +1}}\left(\begin{matrix}
\Pi_{h}\Pi_h &\Pi_{h+1}\Pi_{ h+1}
&\cdots&\Pi_{t+h} \Pi_{t+h}
\\ \Pi_h\Pi_{h-1}  &\Pi_{h+1} \Pi_{h}
&\cdots&\Pi_{t+h}\Pi_{ t+h-1}\\\vdots&\vdots&\ddots&\vdots\\\Pi_h
&\Pi_{h+1}\Pi_1&\cdots &\Pi_{t+h}\Pi_t
\end{matrix}\right)
\eeao
and $\bfE_t$ is independent of $|Y_0|$. The \rhs\ can be identified
with $(\Theta_0,\ldots,\Theta_t)$.
\par
An application of Theorem~\ref{main} now yields a stable limit result
for the sample autocovariance  \fct\ of $(X_t)$: Assume that 
$(a_n)$ satisfies $n\,\P(|X_0 \Phi_0|>a_n)\sim 1$. In view of
\eqref{eq:iter} and Breiman's result (see \cite{breiman:1965}) we also
have 
\beao
n \P(|X_0 \Phi_0|>a_n)\sim n \P( X^2>a_n)
\, \E\Big[\Big(|\Pi_h|\sqrt{1+\Pi_1^2+\cdots + \Pi_h^2}\big)^{\alpha/2}\Big]\,.
\eeao
In view of Kesten's result \cite{kesten:1973}, $\P(|X|>x))\sim c_0
x^{-\alpha}$. Therefore we can choose 
\beao 
a_n= n^{2/\alpha} \big(c_0 \E\Big[\Big(|\Pi_h|\sqrt{1+\Pi_1^2+\cdots +
  \Pi_h^2})^{\alpha/2}\Big]\Big)^{2/\alpha}
\,.
\eeao
Then we have for $m\ge 0$,  $\alpha\in (2,4)$,
\beao
\Big(a_n^{-1}\sum_{t=1}^{n-h} \big(X_tX_{t+h}- 
\E(X_0 X_h)\big)\Big)_{h=0,\ldots,m}\std \xi_{\alpha/2}\,,
\eeao
and for $\alpha\in (0,2)$,
\beao
\Big(a_n^{-1}\sum_{t=1}^{n-h} X_tX_{t+h}\Big)_{h=0,\ldots,m}\std \xi_{\alpha/2}\,,
\eeao
where $\xi_{\alpha/2}$ is an $\alpha/2$-stable $\bbr^{h+1}$-valued
random vector whose  \chf\ is given in
Theorem~\ref{main} and $(\Theta_t)_{t\ge 0}$ is described above.
This result was proved in Basrak et
al. \cite{basrak:davis:mikosch:2002a},
Theorem 2.13. In the case $\alpha\in (2,4)$ the additional condition
(2.20) was needed; the latter condition is hardly verifiable and could
be overcome in the present paper by showing condition ${\bf (DC_{\it
    p})}$.
Moreover, as in \cite{basrak:davis:mikosch:2002a} a straightforward
application of the \cmt\ yields a corresponding limit result for the
sample autocorrelation \fct ; we omit details. 
The limit laws in Theorem 2.13 of \cite{basrak:davis:mikosch:2002a} are expressed in terms of the points
of the limiting \pp es in Theorem~\ref{thm:davhs} above, while our
limits are expressed in terms of the cluster index $b$. Neither of the 
\rep s of the $\alpha$-stable limits are easy due to the complicated
dependence structure.

\subsection{Sample mean of a \garch\ process and its volatility,
  sample covariance \fct\ of a \garch\ process }\label{exam:garch}
We consider a \garch\ process $X_t=\sigma_t\,Z_t$, where
$(Z_t)$ is an iid \seq\ of mean zero unit variance \rv s and 
$(\sigma_t)$ is a \seq\ of  non-negative \rv s \st\ 
$\sigma_t^2 = \alpha_0 + \sigma_{t-1}^2 (\alpha_1 Z_{t-1}^2+\beta_1)$.
Here $\alpha_0,\alpha_1,\beta_1$ are positive constants. The latter
equation is of Kesten type \eqref{eq:sre} with 
$A_t=\alpha_1 Z_{t-1}^2+\beta_1$ and $B_t=\alpha_0$. We assume that
the conditions of Section~\ref{exam:1} are satisfied, in particular,
\beao
\P(\sigma>x)\sim c_0 x^{-\alpha}\,,\quad \xto\,,
\eeao
for some constant $c_0>0$ and tail index $\alpha>0$, satisfying 
$\E (\alpha_1 Z_{0}^2+\beta_1)^{\alpha/2}=1$. We also assume that 
$\E|Z|^{\alpha+\epsilon}<\infty$ for some $\epsilon>0$. 
Rewriting \eqref{eq:sre2}, we have 
\beao
(\sigma_0^2,\ldots,\sigma_t^2)=\sigma_0^2 (1,\Pi_1,\ldots,\Pi_t)+ \bfR_t\,,
\eeao
where $\E|\bfR_t|^{(\alpha+\epsilon)/2}<\infty$ and also 
$\E|\Pi_i|^{(\alpha+\epsilon)/2}<\infty$ for $i\ge 1$. An application of Breiman's 
multivariate  result (see Basrak et
al. \cite{basrak:davis:mikosch:2002}) shows that for any continuity
set $M$ as $\xto$,
\beao
\dfrac{\P(x^{-1}(\sigma_0,\ldots,\sigma_t)\in M)}{\P(\sigma >x)}
&\sim& 
\dfrac{\P(x^{-1}\sigma_0(1,\Pi_1^{0.5},\ldots,\Pi_t^{0.5})\in
  M)}{\P(\sigma >x)}\\
&\to& 
\int_0^\infty \alpha y^{-\alpha-1} P(y(1,\Pi_1^{0.5},\ldots,\Pi_t^{0.5})\in
  M) \,dy\,.
\eeao
This shows that \regvar\ of $(\sigma_t)$ with index $\alpha$ follows
from the \regvar\ of $\sigma$. This
property is inherited by the \seq\ $(X_t)$. We observe that as $\xto$,
\beao
\dfrac{\P(|(X_0,\ldots,X_t)-\sigma_0
(Z_0,\Pi_1^{0.5}
Z_1,\ldots,\Pi_t^{0.5}Z_t)|>x)}{\P(\sigma>x)}&\le &
\dfrac{\P( |Z_1| R_1^{0.5}+ \cdots + |Z_t| R_t^{0.5}>x)}{\P(\sigma>x)}=o(1)\,.
\eeao
In the last step we used the independence of $Z_i$ and $R_i$ as well
as the moment condition on $Z$. Condition ${\bf (RV_\alpha)}$ for
$(X_t)$ now follows. This property was proved in 
Mikosch and \sta\ \cite{mikosch:starica:2000} under the additional
condition that $Z$ be symmetric. The above calculation shows that this
assumption can be avoided.
\par
Next we consider the 2-dimensional \MC\
\beao
\Phi_t=(\sigma_t,X_t)'=\sigma_t (1,Z_t)'\,,\quad t\in \bbz\,.
\eeao
A similar calculation as above shows that this \MC\ satisfies ${\bf
  (RV_\alpha)}$ and for $h\ge 0$, any continuity set $N$,
observing that $|\Phi_0|=\sigma_0\sqrt{Z_0^2+1}$,
\beao\lefteqn{
\P( x^{-1}(\Phi_0,\ldots,\Phi_h)\in N\mid |\Phi_0|>x)}\\
&\sim & \P\big( x^{-1} 
\sigma_0\big((1,Z_0)', \Pi_1^{0.5}(1,Z_1)',\ldots,\Pi_{h}^{0.5}
(1,Z_h)'\big) \in N\mid |\Phi_0|>x\big)\\
&\stw& \P( |Y_0|
\big((1,Z_0)',\Pi_1^{0.5}(1,Z_1)',\ldots,\Pi_{h}^{0.5}
(1,Z_h)'\big)/(Z_0^2+1)^{0.5}\in N)\,.
\eeao
Identifying the limiting vector with $|Y_0|(\Theta_0,\ldots,
\Theta_h)'$, we have for any $\theta\in \bbs$,
\beao
b(\theta) =\E\Big[\Big\{\Big(\theta'(1,Z_0)'+ 
\sum_{t\ge 1}\Pi_{t}^{0.5} \theta' (1,Z_t)'\Big)_+^{\alpha}-\Big( 
\sum_{t\ge 1}\Pi_{t}^{0.5} \theta' (1,Z_t)'\Big)_+^{\alpha}\Big\}\Big/(Z_0^2+1)^{\alpha/2}\Big]\,.
\eeao
\par
The \MC\ $(\Phi_t)$ is aperiodic and irreducible under 
classical conditions on the density of the $Z$; see
e.g. \cite{mikosch:starica:2000} for details. The condition 
${\bf (DC_{\it p})}$ for $p<\alpha$ follows by an 
application of Lemma \ref{lem:skel} for $V(x)=|x|^p$. We recall that
for $m\ge 2$,
$\sigma_m^2=\Pi_{2,m}(\alpha_0+\alpha_1 X_0^2+\beta_1 \sigma_0^2)+
 \wt R_m$, where $\E|\wt R_m|^{(\alpha+\epsilon)/2}<\infty$ for some $\epsilon>0$
and  $\wt R_m$ is independent of $Z_m$.
We have for $p<\alpha$, some $c>0$,
\beam\label{eq:999}
\E[|\Phi_m|^p\mid \Phi_0=\bfy] 
&=&  \E|\Pi_{2,m}(\alpha_0+\alpha_1 y_1^2+\beta_1 y_2^2)+\wt R_m|^{p/2}\E (Z^2+1)^{p/2}\nonumber\\
&\le &  |\bfy|^p \E|\Pi_{2,m} |^{p/2}\E (Z^2+1)^{p/2} \max(\alpha_1^p,\beta_1^p)+c\,.
\eeam
For $m=1$, we find constants $c_1,c_2>0$ \st\ 
 $\E(V(\Phi_1)|\Phi_0=\bfy)\le c_1 V(\bfy)+c_2$ .
Since $\E A^{p/2}<1$ for $p<\alpha$,  ${\bf (DC_{\it p,m})}$ holds for
sufficiently large $m$ in view of \eqref{eq:999}.  An application of Lemma
\ref{lem:skel} concludes the proof. Thus we may apply the stable limit
theory of Theorem~\ref{main} with $f(x)=x$ to $(\Phi_t)$ for $\alpha<2$ and the limit
law is determined by the cluster index $b$ above.
\par
For $h\ge 0$ consider the \MC , recycling the notation $\Phi_t$,
\beam\label{eq:mcnew}
\Phi_t&=&(X_t,\sigma_t,\ldots,X_{t-h},\sigma_{t-h})\,,\quad t\in
\bbz\,.
\eeam
We also write
\beao
\Phi_t^2&=&(X_t^2,\sigma_t^2,\ldots,X_{t-h}^2,\sigma_{t-h}^2)\,,\quad t\in\bbz\,,
\eeao
and introduce the \fct\ $f$ acting on $(\Phi_t)$ given by
\beao
\bfY_t=f(\Phi_t)= \big(X_t (X_{t-1},\ldots,X_{t-h}), \Phi_t^2,
\Phi_t\big)\,,\quad t\in \bbz\,.
\eeao
We intend to show  ${\bf (DC_{\it p,m})}$ for $p<\alpha/2$ and some large $m$. We restrict ourselves to
the case $h=1$; the general case is analogous but requires more 
accounting.
We observe that for suitable constants $c>0$,
\beao
|f(\Phi_t)|^p &=&|X_t^2 X_{t-1}^2+ X_{t}^4+X_{t-1}^4+\sigma_t^4+
\sigma_{t-1}^4+
X_t^2+X_{t-1}^2+\sigma_t^2+
\sigma_{t-1}^2 |^{p/2}\\
&\le &c \big((1+Z_t^4) (1+X_{t-1}^4+\sigma_{t-1}^4)
+Z_t^2 (1+\sigma_{t-1}^2+X_{t-1}^2) (1+X_{t-1}^2)+X_{t-1}^2+1+\sigma_{t-1}^2\big)^{p/2}
\eeao
Then for suitable constants $c_1,c_2>0$,
\beao
\E[|f(\Phi_1)|^p\mid \Phi_{0}=\bfy]
&\le & c\big(1+ |y_1^2|^p+ |y_2^2|^p+  |y_1|^p+ |y_2|^p\big)
\big)\\
&\le & c_1 |f(\bfy)|^p+c_2\,.
\eeao
By a similar argument, for sufficiently large $m\ge 1$, suitable
constants $c>0$,
recalling that $\sigma_t^2=\Pi_t\sigma_0^2+R_t$, where $\sigma_0^2$ is
independent of $(\Pi_t,R_t)$, and 
\beao
\E[|f(\Phi_m)|^p\mid \Phi_{0}=\bfy]
&\le & c\big(1+\E[|\sigma_{m-1}^4|^{p/2}+|\sigma_{m-1}^2|^{p/2}\mid
\Phi_{0}=\bfy]\big)\\
&\le & c\big(1+ \E\Pi_m^{2p}\,|y_2|^{2p} + \E\Pi_m^{p}]\,
|y_2|^{p}\big)\\
&\le & c\,\big(\E[\Pi_m^{2p}] + \E[\Pi_m^{p}]\big)\,|f(\bfy)|^p+c\\
&\le &\beta |f(\bfy)|^p+c\,,
\eeao
for some $\beta\in (0,1)$, sufficiently large $m\ge 1$. Here we used
the fact that $\E A ^{2p}<1$ for $p<\alpha/2$. Now we can apply 
Lemma~\ref{lem:skel} to show ${\bf (DC_{\it p})}$ for $p<\alpha/2$ 
\par
It remains to show ${\bf (RV_{\alpha/{\rm 2}})}$ for $(\bfY_t)$ defined in
\eqref{eq:mcnew}. The $\Phi_t$-component of $\bfY_t$ is \regvary\ with
index $\alpha$.  Therefore, without loss of generality and abusing notation, 
we will consider the \seq\ 
\beao
\bfY_t=f(\Phi_t)= \big(X_t (X_{t-1},\ldots,X_{t-h}), \Phi_t^2)\,,\quad
t\in \bbz\,.
\eeao
Similar arguments as in the first part of this subsection 
and as in Section~\ref{exam:1} show for $t\ge 0$ that
\beao
&&\bfY_t=  \bfR_t^{(1)}+\\&&\sigma_{t-h}^2 
\big(Z_t \Pi_{t-h+1,t}^{0.5}(Z_{t-1}\Pi_{t-h+1,t-1}^{0.5},\ldots,
Z_{t-h} ), (\Pi_{t-h+1,t}(Z_t^2,1),\ldots, (Z_{t-h}^2,1))
\big)'\\&&= \bfR_t^{(2)}+\\&&\sigma_{-h}^2 \Pi_{1-h,t-h} 
\big(Z_t \Pi_{t-h+1,t}^{0.5}(Z_{t-1}\Pi_{t-h+1,t-1}^{0.5},\ldots,
Z_{t-h} ), (\Pi_{t-h+1,t}(Z_t^2,1),\ldots, (Z_{t-h}^2,1))
\big)'\,,
\eeao
where $\E | \bfR_t^{(i)}|^{(\alpha+\vep)/2}<\infty$, $i=1,2$.
Therefore 
\beao
(\bfY_0,\ldots,\bfY_t)'=\wt \bfD_t\sigma_{-h}^2
 + \wt\bfQ_t\,,
\eeao
where $\E | \wt \bfQ_t|^{(\alpha+\vep)/2}<\infty$ and  $\E | \wt
\bfD_t|^{(\alpha+\vep)/2}<\infty$ for some $\vep>0$
and
\beao
\wt\bfD_t= \left(
\begin{matrix} 
Z_0Z_{-1}\Pi_{1-h,0}^{0.5}\Pi_{1-h,-1}^{0.5}& 
Z_1Z_{0}A_{1-h}\Pi_{2-h,1}^{0.5}\Pi_{2-h,0}^{0.5}&
\cdots&
Z_tZ_{t-1}\Pi_{1-h,t-h}\Pi_{t-h+1,t}^{0.5}\Pi_{t-h+1,t-1}^{0.5}\\
Z_0Z_{-2}\Pi_{1-h,0}^{0.5}\Pi_{1-h,-2}^{0.5}&
Z_1Z_{-1}A_{1-h}\Pi_{2-h,1}^{0.5}\Pi_{2-h,-1}^{0.5}&
\cdots &
Z_tZ_{t-2}\Pi_{1-h,t-h}\Pi_{t-h+1,t}^{0.5}\Pi_{t-h+1,t-2}^{0.5}\\ 
\vdots&\vdots&\ddots&\vdots\\
  Z_0Z_{-h}\Pi_{1-h,0}^{0.5}&
Z_1Z_{1-h}A_{1-h}\Pi_{2-h,1}^{0.5}&
\cdots &
Z_tZ_{t-h}\Pi_{1-h,t-h}\Pi_{t-h+1,t}^{0.5} \\
\Pi_{1-h,0}(Z_0^2,1)&A_{1-h}\Pi_{2-h,1}(Z_1^2,1)&\cdots &\Pi_{1-h,t-h}\Pi_{t-h+1,t}(Z_t^2,1)\\
\Pi_{1-h,-1}(Z_{-1}²,1)&A_{1-h}\Pi_{2-h,0}(Z_0^2,1)&\cdots &\Pi_{1-h,t-h}\Pi_{t-h+1,t-1}(Z_{t-1}^2,1)\\
\vdots&\vdots&\ddots&\vdots\\
(Z_{-h}^2,1)&A_{1-h} (Z_{1-h}^2,1)&\cdots &\Pi_{1-h,t-h} (Z_{t-h}^2,1)
\end{matrix}
\right)\,.
\eeao
Notice that 
$\sigma_{-h}^2$ and $\wt \bfD_t$ are independent and that
$\sigma_{-h}^2$
is \regvary\ with index $\alpha/2$.
Then 
${\bf (RV_{\alpha/{\rm 2}})}$ for $(\bfY_0,\ldots,\bfY_t)$ follows by an application of
the  multivariate Breiman result; see
\cite{basrak:davis:mikosch:2002}. We omit the calculation of the
cluster index; it is similar to its calculation in
Section~\ref{exam:1}.
\par
Now we can apply Theorem~\ref{main} to prove limit theory with
$\alpha/2$-stable limits, $\alpha<4$, for the sample autocovariance
\fct\ of the \garch\ process. The corresponding theory using 
\pp\ techniques is given in
\cite{davis:mikosch:1998,mikosch:starica:2000}. There the limit
theory for the \seq s $(|X_t|)$ and $(X_t^2)$ was also provided. The same
results can be provided by Theorem~\ref{main} by calculating the
corresponding cluster indices.   Applied to the squares $(X_t^2)$ we obtain in particular for $\alpha\in (2,4)$,
\beam\label{eq:ac}
na_n^{-1}\frac 1n\sum_{t=1}^{n-h} X_t^2X_{t+h}^2-\Big(\frac1n\sum_{t=1}^{n} X_t^2\Big)^2\std \xi_{\alpha/4}\,,
\eeam
where $\xi_{\alpha/4}$ is an $\alpha/4$-stable \rv\ whose  \chf\ is given in
Theorem~\ref{main} and $(\Theta_t)_{t\ge 0}=(cZ_t^2Z_{t+h}^2\Pi_t\Pi_{t+h})_{t\ge 0}$ for some $c>0$. In particular, the $\Theta_t$s are non negative   and thus $b_-=0$. Then $\xi_{\alpha/4}$ is supported on $[-(\E X_0^2)^2,\infty)$. We omit further details. Relation \eqref{eq:ac} supports the idea of spurious long-range dependence effects observed on real-life log-return data which are often observed to have infinite fourth moments; see \cite{mikosch:starica:2003} for a discussion.

\section{Proof of Theorem~\ref{main}}\label{sec:mainproof}\setcounter{equation}{0} 
\subsection{Proof of part (1)} We will use the Cram\'er-Wold device to show that
$(a_n^{-1} \theta 'S_n)$ has a  (possibly degenerate)
$\alpha$-stable limit $\xi_\alpha(\theta)$ for every $\theta\in
\bbs^{d-1}$.
We will apply Theorem 1 in 
\cite{bartkiewicz:jakubowski:mikosch:wintenberger:2011} which we
recall for convenience:
\bth\label{thm:clt} Assume that $(G_t)$ is a strictly
stationary process of \rv s, satisfying the following conditions.
\begin{enumerate}
\item
The \regvar\ condition ${\bf (RV_\alpha)}$ holds for some $\alpha\in (0,2)$.
\item
The mixing condition {\bf (MX)}: There exist $m=m_n\to\infty$ \st\
$k_n=[n/m_n]\to\infty$ and
\beao
\E\ex^{it b_n^{-1}S_n(G)}- \Big(\E\ex^{it
  b_n^{-1}S_m(G)}\Big)^{k_n}\to 0\,,\quad \nto\,,\quad  t\in\bbr\,,
\eeao
where $S_n(G)=G_1+\cdots +G_n$ and $(b_n)$ is chosen \st\
$n\,\P(|G_1|>b_n)\sim 1$.
\item
The anti-clustering condition
\begin{equation}\label{AC}\tag{\bf AC}
\lim_{\ell\to\infty}\limsup_{n\to\infty} \dfrac n
m\sum_{j=\ell+1}^{m}\E\left|\overline{
    x\,b_n^{-1}(S_{j}(G)-S_\ell(G))}\;\overline{ x\,b_n^{-1}
    G_{1}}\right|=0\,,\quad x\in\bbr\,,
\end{equation}
holds, where $m=m_n$ is the same as in {\bf (MX)} and $\ov x =
(x\wedge 2)\vee (-2)$.
\item
The limits
\begin{equation}\label{TB}
\tag{\bf TB} \lim_{\ell\to \infty}(b_+(\ell)-b_+(\ell-1))=c_+\mbox{ and
}\lim_{\ell\to \infty}(b_-(\ell)-b_-(\ell-1))=c_-\,,
\end{equation}
exist. Here $b_+(\ell),b_-(\ell)$ are the tail balance parameters given by
$b_+(\ell)=\lim_{\nto} n\,P(S_\ell(G)>b_n)$ and $b_-(\ell)=\lim_{\nto} n\,P(S_\ell(G)\le -b_n)$.
\item
For  $\alpha>1$ assume $\E G_1=0$ and for $\alpha=1$,
\begin{equation}\label{CT}\tag{\bf CT}
\lim_{\ell\to\infty}\limsup_{n\to\infty}n\,|\E(\sin(b_n^{-1}S_\ell(G)))|=0.
\end{equation}
\end{enumerate}
Then $c_+$ and $c_-$ are non-negative and $(b_n^{-1}S_n(G))$ converges
in \ds\ to an $\alpha$-stable \rv\ (possibly zero) with \chf\
$\psi_{\alpha}(x) = \exp( -|x|^{\alpha} \chi_{\alpha}(x, c_+,
c_-))$, where for $\alpha \neq 1$  the function $\chi_\alpha(x,c_+,
c_-), x \in \bbr$, is given by the formula
\[\dfrac{\Gamma(2-\alpha)}{ 1-\alpha }\,\Big((c_++c_-)\,\cos(\pi
\alpha/2)-i\,\sign(x) (c_+-c_-)\, \sin(\pi\, \alpha/2)\Big)\,,\]
while for $\alpha = 1$ one has
\[\chi_1(x,c_+, c_-) =
0.5\,\pi (c_++c_-) +i\,\sign(x)\,(c_+-c_-) \log |x|,\quad x\in \bbr.
\]
\ethe
We will verify the conditions of this theorem for the \seq\
$G_i=\theta'X_i$ for fixed $\theta\in \bbs^{d-1}$.
\subsubsection*{The \regvar\ condition ${\bf (RV_\alpha)}$ for $(G_t)$}
This condition is straightforward from the definition of  ${\bf
  (RV_\alpha)}$ for $(X_t)$ and the fact that the \fct\ $f(x)=\theta'
x$, $x\in \bbr^d$, is continuous and homogeneous.

\subsubsection*{The anti-clustering condition {\bf (AC)}}
Without loss of generality we assume that  ${\bf (DC_{\it p})}$ holds
for $V(y)=|y|^{p}$. We also assume $p\le 1$; for $p>1$ an application
of Jensen's inequality yields  ${\bf (DC_{\it p'})}$ for any $p'<p$.
Since $p\le 1$ there exists $c>0$
such that $y\le c\,y^ p$ for $y\in [0,2]$. Then one has
\beao
\lefteqn{T_{\ell m}=\dfrac{n}m\sum_{j=\ell+1}^{m}\E\Big[ \overline{\big|
    x\,b_n^{-1}(S_{j}(G)-S_ \ell(G))\big|}\;\overline{| x\,b_n^{-1}
    G_{1}|}\Big]}\\ &\le& c \frac nm \,\sum_{j=\ell+1}^{m}\E\Big[\overline{\big|
    x\,b_n^{-1}(S_{j}(G)-S_\ell(G))\big|^p}\;\overline{ \big| x\,b_n^{-1}
    G_{1}\big|}\Big]\,.
\eeao
Using  ${\bf (DC_{\it p})}$, a recursive argument yields
\beam\label{rec:dcp}
\E( |G_k|^{p} \mid \Phi_1=y)\le
\beta^{k-1} |f(y)|^{p}+b\,\sum_{j=1}^{k-1}\beta^{j}\,,\quad k\ge 2\,,
\eeam
where $\beta,b$ appear in  ${\bf (DC_{\it p})}$.
Multiple use of this argument
and the subadditivity of the \fct\ $z\mapsto \overline {z}$ on
$(0,\infty)$  yield for $\ell<j\le m$,
$$
\E\Big[\overline{\big| xb_n^{-1}(S_j(G)-S_{\ell}(G)) \big|^{p}}\mid
\Phi_1\Big]\le c\,\overline{|x|^{p}b_n^{-p}\sum_{i=\ell+1}^m\beta^i|X_1|^{p}}+
c b_n^{-p} \,m\,.
$$
Conditioning on $\Phi_1$, the latter inequality finally yields
\beao
\E T_{\ell m}\le c
\dfrac n
m\sum_{j=\ell+1}^{m}\E\Big[\overline{|x|^{p}b_n^{-p}
\sum_{i=1}^j\beta^i|X_1|^p}\;\overline{ x\,b_n^{-1}
    |X_{1}|}\Big]+
c\, \dfrac{ \,m\,n}{b_n^p} \;\E\overline{|xb_n^{-1}X_1|} =I_1+I_2\,.\eeao
We have 
$I_2 \le c b_n^{-p-1} n\, m=o(1)$
if we choose $m=m_n=\log^2 n$.
It remains to prove that $I_1$
is \asy ally negligible. An application of Karamata's theorem
yields the bound
$$
I_1\le  c\,\dfrac n
m\sum_{j=\ell+1}^{m}\P\Big(|X_1|\ge c b_n(\sum_{i=\ell}^j\beta^i)^{-1/(p+1)}\Big)\le  \frac cm \sum_{j=\ell+1}^{m}(\sum_{i=\ell}^j\beta^i)^{\alpha/(p+1)}\le c \beta^{\ell\alpha/(p+1)}.
$$
The \rhs\ vanishes as  $\ell\to\infty$. Collecting the above bounds,
condition  {\bf (AC)} follows.
\subsubsection*{The mixing condition {\bf (MX)}}
Here we give a significant improvement on
Lemma 3 in 
\cite{bartkiewicz:jakubowski:mikosch:wintenberger:2011}; in the latter
paper it is assumed that $(G_t)$ is strongly mixing. The next result
avoids this condition.
\ble\label{lem:known} Consider a strictly stationary real-valued \seq\ $(G_t)$
satisfying the conditions
${\bf (RV_\alpha)}$ for some $\alpha\in (0,2)$ and
{\bf (AC)}. Then {\bf (MX)} can be replaced by\\
Condition {\bf (MX')}:
There exists a sequence $(r_n)$ \st\ $r_n=o(m_n)$ and
\beao
|\varphi_n^{(\ell)}(t)-\varphi_{n,m-\ell}^k(t)|\to 0\,,\qquad  t\in\R\,.
\eeao
holds for $\ell=m-r_n$ and $\ell=r_n$, where
\beao
\varphi_n^{(\ell)}(t)&=&\E\Big[\exp\Big(itb_n^{-1}\sum_{i=1}^{k_n}\sum_{t=(i-1)m+1}^{im-\ell}G_t\Big)\Big]\,,\\
\varphi_{n,j}&=&
\E\Big[\exp\Big(itb_n^{-1}\sum_{t=1}^{j}G_t\Big)\Big]\,,\quad
j\ge 1\,,\quad \varphi_n(t)=\varphi_{n,n}(t)\,,\quad t\in \bbr\,.
\eeao
\ele
\begin{proof}
Notice that condition {\bf (MX)} can be written in the form
$\varphi_n(t)-\varphi_{n,m}^k(t)\to 0$ as $\nto$.
We have
\beao
\varphi_n(t)-\varphi_{n,m}^k(t)&=&
[\varphi_n(t)-\varphi_n^{(r)}(t)]
+[\varphi_n^{(r)}(t)-\varphi_{n,m-r}^k(t)]+
[\varphi_{n,m-r}^k(t)-\varphi_{n,m}^k(t)]\\&=&
P_1+P_2+P_3\,.
\eeao
In view of {\bf (MX)'}, $P_2\to 0$. Next we deal with $P_1$.
Assume for
simplicity that $k_n=n/m$ is an integer.
We use the classical Bernstein blocks technique, writing
$$
S_n=b_n^{-1}\sum_{i=1}^{k_n}\sum_{t=(i-1)m+1}^{im-r}G_t+b_n^{-1}\sum_{i=1}^{k_n}\sum_{t=im-r+1}^{im}G_t=I_1+I_2\,.
$$
We will show that $\E\exp(itI_2)\to 1$.
Condition {\bf (MX)'}
implies that
$|\E\exp(itI_2)-\varphi_{n,r}^k(t)|\to 0$ as $\ell=m-r\ge r$ and
$\ell /n\to 0$. Moreover, Lemma 3.5 in \cite{petrov:1995}
yields that $\varphi_{n,r}^k(t)\to 1$ \fif\  $k(\varphi_{n,r}(t)-1)\to
0$. Assuming ${\bf (RV_\alpha)}$ and {\bf (AC)}, one can follow
the proof of Lemma 1 in 
\cite{bartkiewicz:jakubowski:mikosch:wintenberger:2011}.
We have
$$
\lim_{q\to \infty}\limsup_{n\to\infty}|k\,(\varphi_{n,r}(t)-1)-
k\,r\,(\varphi_{n,q}(t)-\varphi_{n,q-1}(t))|\to0,\quad t\in\R.
$$
Under ${\bf (RV_\alpha)}$, an application of
Theorem 3 in Section XVII.5 of Feller gives that
$n(\varphi_{n,q}(t)-1)$ converges for all $q$.
We deduce that $n(\varphi_{n,q}(t)-\varphi_{n,q-1}(t))$
converges too. As $kr/n\sim r/m\to0$ we
conclude that $kr(\varphi_{n,q}(t)-\varphi_{n,q-1}(t))\to0$
and then $k_n(\varphi_{n,r}(t)-1)\to0$ which gives the desired result
$\E\exp(itI_2)\to 1$, equivalently, $I_2\stp 0$. Since
\beao
|P_1|= \Big|\E\Big[\exp(it (I_1)( 1-\exp(it I_2)))\Big]\Big|\le
\E\Big|1-\exp(it I_2)
\Big|\,,
\eeao
dominated \con\ yields $P_1\to 0$. Finally,
\beao
|P_3|\le k\,\Big|(\varphi_{n,m-r}(t)-1)-
(\varphi_{n,m}(t)-1)\Big|\to 0\,.
\eeao
and the same arguments as above show that $P_3\to 0$.
\end{proof}
We finish the proof of {\bf (MX)} for the \seq\ $(G_t)$. In view
of   ${\bf (DC_{\it p})}$, $(X_t)$, hence $(G_t)$, are $\beta$-mixing, hence
strongly mixing, with
exponential rate $(\alpha_h)$. We will show {\bf (MX)} by an application of
Lemma~\ref{lem:known}. A standard telescoping sum argument shows that
\beao
|\varphi_n^{(\ell)}(t)-\varphi_{n,m-\ell}^k(t)|&\le c\,k_n \alpha_\ell\,.
\eeao
Since we choose $m=\log^2 n$ in the proof of {\bf (AC)}, $k_n  \alpha_\ell\le (n/\log^2 n) \exp(-c
\ell_n)$. Thus, choosing $\ell_n= C\log n$ for some sufficiently large
constant $C>0$ we have $\ell_n=o(m_n)$, $k_n  \alpha_\ell=o(1)$
and we can also find $r_n=o(\ell_n)$. This proves {\bf (MX')}, hence
{\bf (MX)}.

\subsubsection*{Condition ${\bf (TB)}$ }
Note that  $\{|\theta ' X|
>b_n\}\subset\{|X|>b_n\} $. Then
\beao
b_+(\ell)&=& \lim_{\xto} \dfrac{\P(S_\ell(G)>x)}{\P(|\theta'X|>x)}\\&=&
\lim_{\xto} \dfrac{\P(\theta 'S_\ell>x)}{\P(|X|>x)}\lim_{\xto
}\dfrac{\P(|X|>x)}{\P(|\theta'X|>x)}\\
&=& b_\ell(\theta) \,\lim_{\xto
} (\P(|\theta'X|>x\mid |X|>x))^{-1}\\
&=& b_\ell(\theta)  (\P(|Y_0| |\theta'\Theta_0|>1)))^{-1}\\
&=&b_\ell(\theta)  (\E(|\theta'\Theta_0|^\alpha))^{-1}
\,.
\eeao
Correspondingly, $b_-(\ell)=b_\ell(-\theta)(\E(|\theta'\Theta_0|^\alpha))^{-1}$. Here we assumed that
$\E(|\theta'\Theta_0|^\alpha)\ne 0$. Otherwise,
$b_+(\ell)=b_-(\ell)=0$. \\[1mm]
Thus we may apply Theorem~\ref{thm:clt} to
conclude that
$
b_n^{-1} \theta'S_n \std \xi_\alpha(\theta)
$ for an $\alpha$-stable \rv\ $\xi_\alpha(\theta)$ with \chf\
$\psi_\alpha(x,\theta)$ given by
\beao\lefteqn{
\E(|\theta'\Theta_0|^\alpha)\,\log \psi_\alpha(x,\theta)}\\
&=& -|x|^{\alpha}
\dfrac{\Gamma(2-\alpha)}{ 1-\alpha}\,\Big((b(\theta)+ b(-\theta))\,\cos(\pi
\alpha/2)-i\,\sign(x) (b(\theta)- b(-\theta))\, \sin(\pi\,
\alpha/2)\Big)\,,\quad
x\in\bbr\,.
\eeao
The factor $\E(|\theta'\Theta_0|^\alpha)$ on the \lhs\ is due to the
normalization $(b_n)$ instead of $(a_n)$. Replacing $(b_n)$ by $(a_n)$,
we have for any $v\in\bbr^d$ that
\beao
&&\E \ex^{i v' (a_n^{-1}S_n)}\to\\
&&\exp\left\{ - |v |^{\alpha}
C_\alpha^{-1}\,\Big((b(v/|v|)+b(-v/|v|))\,-i\, (b(v/|v|)-
b(-v/|v|))\, \tan(\pi\, \alpha/2)\Big)\right\}\,,
\eeao
where $C_\alpha$ is defined in  \eqref{eq:calpha}.
This is the \chf\ of an $\alpha$-stable random vector $\xi_\alpha$.
The \rep\  of the \levy\ spectral measure $\Gamma_\alpha$ in \eqref{eq:b1}
follows by calculations as in Example 2.3.4 of
\cite{samorodnitsky:taqqu:1994}. Indeed, keeping notations of \cite{samorodnitsky:taqqu:1994} and identifying the limiting law yields the equations
\beao
b(\theta)+b(-\theta)=C_\alpha\,\sigma_\theta^\alpha=C_\alpha\,\int_{\bbs^{d-1}}
|\theta's|^\alpha\Gamma_\alpha(ds)=C_\alpha\,\int_{\bbs^{d-1}}
(\theta's)_+^\alpha\Gamma_\alpha(ds)+C_\alpha\,\int_{\bbs^{d-1}}
(-\theta's)_+^\alpha\Gamma_\alpha(ds)
,\eeao
and
\beao 
b(\theta)-b(-\theta)&=&(b(\theta)+b(-\theta))\,\beta_\theta\\
&=&C_\alpha\,\int_{\bbs^{d-1}}
|\theta's|^\alpha\sign(\theta's)\Gamma_\alpha(ds)\\
&=&C_\alpha\,\int_{\bbs^{d-1}}
(\theta's)_+^\alpha\Gamma_\alpha(ds)-C_\alpha\,\int_{\bbs^{d-1}}
(-\theta's)_+^\alpha \Gamma_\alpha(ds)\,,\quad \theta\in\bbs^{d-1}\,.
\eeao
The limiting $\alpha$-stable distribution
is degenerate \fif\ $b(\theta)=0$ for all $\theta\in \bbs^{d-1}$.
\par
This proves part (1) of the theorem.

\subsection*{Stable limit theory for general \regvary\ stationary
  processes}
In this part we want to give some arguments showing that the
  results of Theorem~\ref{main} can be applied in  much more general context.
For this reason, consider a strictly stationary 
$\bbr^d$-valued \regvary\ \seq\ $(X_t)$ with index $\alpha>0$.
Then $\Phi_t=(X_t,X_{t-1},\ldots)$, $t\in\bbz$, constitutes a \MC\
with infinite-dimensional state space.
In this setting,   ${\bf (DC_{\it p})}$ for $X_t=f(\Phi_t)$
takes on the form: \\[1mm]
{\em Condition}  ${\bf (DC_{\it p}')}$: 
$$
\E(|X_1|^p\mid| (X_0,X_{-1},\ldots)=(x_0,x_{-1},\ldots))\le \beta 
|x_0|^p+b\quad \mbox{for some $0<\beta<1$ and $b>0$.}
$$
We also need a weak dependence assumption more general 
than geometric  $\beta$-mixing which, in the irreducible case, is 
implied by the drift condition.\\[1mm]
{\em Condition} {\bf (MX$_m$)} :
Consider an integer \seq\ $(m_n)$  \st\ $m=m_n\to\infty$ and
$m_n/n=o(1)$ and also write $k_n=[n/m]$.
There exists a sequence $(r_n)$ \st\ $r_n=o(m_n)$ and
\beao
\lim_{\nto}|\varphi_n^{(\ell)}(s)-\varphi_{n,m-\ell}^k(s)|\to
0\,,\quad  
s\in\R^d\,,
\eeao
holds for both $\ell=\ell_n=m_n-r_n$ and $\ell=r_n$, where
\beao
\varphi_n^{(\ell)}(s)&=&\E\Big[\exp\Big(ia_n^{-1}\sum_{i=1}^{k_n}
\sum_{t=(i-1)m+1}^{im-\ell}s'X_t\Big)\Big]\,,\\
\varphi_{n,j}&=&
\E\Big[\exp\Big(ia_n^{-1}\sum_{t=1}^{j}s'X_t\Big)\Big]\,,\quad
j\ge 1\,,\quad \varphi_n(s)=\varphi_{n,n}(s)\,,\quad s\in \bbr^d\,.
\eeao
\par
Condition {\bf (MX$_m$)} is implied by 
{\em $\theta$-weak dependence} introduced by  
 Doukhan and Louhichi \cite{doukhan:louhichi:1999}: For every $m\ge 1$, equip
$(\R^d)^m$ with the metric 
$|\cdot |_m=m^{-1}\sum_{i=1}^m|\cdot|$. 
A function $f: (\R^d)^m\mapsto [-1,1]$, $m\ge 1$, is {\em Lipschitz} if
$$
\sup_{x\neq y}\frac{|f(x)-f(y)|}{|x-y|_m}=\Lip(f)<\infty.
$$
The $\theta$-weak dependence coefficients $(\theta_r)_{r\ge 0}$ 
are defined for any $f$ with $\Lip(f)=1$ and measurable 
$g: (\R^d)^v\mapsto [-1,1]$, $v\ge 1$, as
$$
\sup_{k,v\ge 1}\sup_{i_1<\cdots <i_v\le 0\le r\le j_1<\cdots<j_m}|\cov(g(X_{i_1},\ldots,X_{i_v}),f(X_{j_1},\ldots,X_{j_m}))|=\theta_r\,.
$$
Condition 
{\bf (MX$_m$)} follows if $\theta_r\to 0$ for some $r=r_n=o(m)$ with $m=m_n$.
$\theta$-weak dependence covers a wide range of known dependence
concepts, including a large variety of mixing conditions; see \cite{doukhan:louhichi:1999}.
\par
In the general case, the following analog of 
Theorem~\ref{main} holds.
The proof follows along the lines of Theorem
\ref{main}. Irreducibility of $(X_t)$ can be replaced by {\bf
  (MX$_m$)}. We omit further details.
\bth\label{thm:nonirr} 
Consider an $\bbr^d$-valued strictly stationary \seq\ $(X_t)$ satisfying
the following conditions:
\begin{itemize}
\item
${\bf (RV_\alpha)}$ for some $\alpha\in (0,2)$, $\E X=0$ if $\alpha>1$ and $X$ is symmetric if $\alpha=1$.
\item
 ${\bf (DC_{\it p}')}$ for some $p\in ((\alpha-1)\vee 0,\alpha)$.
\item 
{\bf (MX$_m$)} for $m_n=o(n^{(p+1)/\alpha-1})$.
\end{itemize}
Let $(a_n)$ be a \seq\ of positive numbers \st\
$n\,\P(|X_0|>a_n)\sim 1$. Then the statement of part (1) of
Theorem~\ref{main}
holds.
\ethe

\subsection{Proof of part (2)} Recall the regenerative structure
of the \MC\ $(X_t)$ from
Section~\ref{subsec:drift}.
We will show that the partial sum $S(1)$ over a full regenerative
cycle is \regvary\ with index $\alpha$. 
We write
\beam\label{eq:sumdecomp}
S_n=S(0)+\sum_{t=1}^{N_A(n)}S(t)+\sum_{i=N_A(n)+1}^nX_i\,,
\eeam
where $N_A(n)=\#\{i\ge 0:\tau_A(i)\le n\}$, $n\ge 1$, is independent of
$(S(i))_{i\ge 1}$.
The first and last block sums $S(0)$ and
$\sum_{i=\tau_A(N_A(n))+1}^nX_i$ are tight. Therefore
$$
a_n^{-1} S_n=a_n^{-1}\sum_{t=1}^{N_A(n)}S(t) + o_P(1)\,.
$$
By virtue of  ${\bf (DC_{\it p})}$ for some $p>0$ the chain $(X_t)$ is
geometrically ergodic. Therefore
there exists a constant $\kappa>0$ \st
\beam\label{eq:tweedie}
\sup_{x\in A}\E_x\ex^{\kappa \tau_A}<\infty\,,
\eeam
(see  \cite{meyn:tweedie:1993},  (15.2) in Theorem~15.0.1) and hence
$\tau_A$ has exponential moment.
By a standard renewal argument, $N_A(n)/ n\stas (\E\tau_A)^{-1}$.
Then for $\epsilon,\delta>0$,
\beao
\lefteqn{\P\Big(a_n^{-1}\Big|\sum_{t=1}^{N_A(n)}S(t)-\sum_{t=1}^{n\,(\E\tau_A)^{-1}}S(t)\Big|\ge
\epsilon\Big) }\\ &\le&\P(
 |N_A(n)-n (\E\tau_A)^{-1}|>\delta N_A(n))\\&&+
 \P\Big(a_n^{-1}\Big|\sum_{t=1}^{|N_A(n)-n (\E\tau_A)^{-1}|}S(t)\Big|\ge
\epsilon\,, |N_A(n)-n (\E\tau_A)^{-1}|\le \delta N_A(n)\Big)\\
&\le & o(1)+ c\P\Big(a_n^{-1}\Big|\sum_{t=1}^{\delta N_A(n)}S(t)\Big|\ge
0.5 \epsilon\Big)\,.
\eeao
In the last step we used a maximal inequality of Ottaviani type; see
e.g. \cite{petrov:1995}, Chapter 2.
The second term on the \rhs\ is neglible, as first letting $\nto$ and then
$\delta\to 0$ since $a_n^{-1}\sum_{t=1}^{N_A(n)}S(t)\std \xi_\alpha$.
Hence
\beao
a_n^{-1}S_n=a_n^{-1}\sum_{t=1}^{n\,(\E\tau_A)^{-1}}S(t)+o_P(1)\,.
\eeao
In view of part (1), the sum of iid random vectors $(S(i))$ on the
\rhs\ has an $\alpha$-stable limit. It follows from
\cite{rvaceva:1962} that $S(1)$ is \regvary\ with index $\alpha$.
This concludes the proof.

\section{Proof of Theorem~\ref{thm:ldp}}\label{sec:proofldp}\setcounter{equation}{0} 
\subsection{Proof of part (1): The case $\alpha\in (0,2)$}\label{sec:smallalpha}
Recall the decomposition \eqref{eq:sumdecomp} of the partial 
sums $S_n$ in terms
of the
regenerative cycles of the \MC . 
We start with an auxiliary result which deals with the sums over the
first and last blocks.
\begin{lemma}\label{prel:lem}
Assume that  ${\bf (RV_\alpha)}$ and  ${\bf (DC_{\it p})}$ hold for some
$p>\alpha-1$ provided $\alpha>1$. Then there exists a constant $c>0$ such
that for any \seq\ $x=x_n\to \infty$  as $\nto$,
\beam\label{eq:33}
\P_A \Big(\sum_{t=1}^{\tau_A}|X_t|>x\Big)&\le& c\, \P(|X|>x)\,,\\
\P \Big(\sum_{t=1}^{\tau_A}|X_t|>x,\tau_A\le n\Big)&=& o(n \P(|X|>x)).\label{eq:34}
\eeam
\end{lemma}
\begin{proof}
We start by proving \eqref{eq:33}.
For any random vector $X$ we write $\ov X= X\1_{\{|X|\le x\}}$.
Then
\beao
\P_A\Big(\sum_{t=1}^{\tau_A}|X_t|>x\Big)\le \P_A\Big(\sum_{t=1}^{\tau_A}\overline {|X_t|}>x/2\Big)+\P_A\Big(\cup_{t=1}^{\tau_A}\{\overline {|X_t|}\neq|X_t|\}\Big)=I_1+I_2.
\eeao
Using the Wald identity, we have
\beao
I_2 =\E_A\Big(\max_{1\le t\le \tau_A}1_{\{|X_t|>x\}}\Big)\le \E_A\Big(\sum_{t=1}^{\tau_A}1_{\{|X_t|>x\}}\Big)= \E_A(\tau_A)\,\P (| X|>x)\,.
\eeao
Write $k_0=\lceil \alpha \rceil$ and choose 
$0<\beta<1$ such that $\beta k_0>\alpha$. Since 
$k_0(k_0-1)\ge \alpha(\alpha-1)$ for $\alpha>1$, we will choose $\beta$ 
such that $p/\beta=k_0-1$.
Markov's inequality yields
\beao
I_1\le c \frac{\E_A\Big(\sum_{t=1}^{\tau_A}\overline{| X_t|}\Big)^{\beta k_0}}{x^{\beta k_0}}\le c \frac{\E_A(\sum_{t=1}^{\tau_A}\overline{ |X_t|}^\beta)^{k_0 }}{x^{\beta k_0}}.
\eeao
Note that $(\overline {|X_t|}^\beta)$ satisfies 
 ${\bf (DC_{\it k_0-1})}$. Under the latter condition we may apply 
Proposition 4.7 of \cite{mikosch:wintenberger:2011} to get
$\E_A(\sum_{t=1}^{\tau_A}\overline{ |X_t|}^\beta)^{k_0 }\le c\E \overline{| X|}^{\beta k_0}$. An application of Karamata's theorem shows that the \rhs\ is bounded by $c \P(|X|>x)$. This concludes the proof of \eqref{eq:33}.
\par
Now we turn to the proof of \eqref{eq:34}.
Abusing notation, we write  $\ov X=\1_{\{|X|\le x\delta\}}$ for any
fixed $\delta$. Then
\beao
\P \Big(\sum_{t=1}^{\tau_A}|X_t|>x,\tau_A\le n\Big)\le
\P\Big(\sum_{t=1}^{\tau_A}\overline {|X_t|}>x/2\,,\tau_A\le n\Big)+\P\Big(\cup_{t=1}^{\tau_A}\{\overline {|X_t|}\neq|X_t|\}\Big)=I_1'+I_2'.
\eeao
Since $\E \tau_A<\infty$ and $X$ is \regvary , we have 
\beao
I_2'\le \E(\tau_A)\,\P (| X|>x\delta)=o(n\P(|X|>x))\,.
\eeao
Similar arguments as above yield
\beao
I_1'\le c \frac{\E(\sum_{t=1}^{\tau_A}\overline {|X_t|}\1_{\{\tau_A\le n\}})^{\beta k_0}}{x^{\beta k_0}}\le c \frac{\E(\sum_{t=1}^{n}\overline{ |X_t|}\1_{\{\tau_A\ge t\}})^{\beta k_0}}{x^{\beta k_0}}.
\eeao
An argument similar to the one used in the proof of Theorem~4.6 
in \cite{mikosch:wintenberger:2011} shows that
\beao
\E\Big(\sum_{t=1}^{n} \ov {|X_t|}^\beta\1_{\{\tau_A\ge t\}}\Big)^{ k_0}\le c\,\E\Big(  \sum_{t=1}^{   n }\overline{|X_t|}^{\beta k_0}\1_{\{\tau_A\ge t\}}\Big).
\eeao
Finally, an application of  Pitman's identity \cite{pitman:1977},  Proposition 4.7 in \cite{mikosch:wintenberger:2011} and Karamata's theorem yield
\beao
\E\Big(  \sum_{t=1}^{   n }\overline {|X_t|}^{\beta k_0}\1_{\{\tau_A\ge t\}}\Big)&=& \P(X_0\in A)\,\E_A\Big(  \sum_{k=0}^{\tau_A-1}\sum_{t=1}^{   n }\overline{| X_{k+t}|}^{\beta  k_0}\1_{\{\tau_A\ge k+t\}}\Big)\\
&\le& n\,  \P(X_0\in A)\,\E_A\Big(\sum_{t=1}^{\tau_A}\overline {|X_{t}|}^{\beta k_0}\Big)
\\
&\le&c\, n   \E \overline {|X|}^{\beta k_0}\sim c\, n\, x^{\beta k_0}\delta^{\beta k_0-\alpha}\P(|X|>x). 
\eeao
Since $\beta k_0>\alpha$ and we can make $\delta$ as small as we like, 
we conclude that
$I_1'=o\big(n\,(\P(|X|>x)\big)$. This concludes the proof of \eqref{eq:34}.
 \end{proof}
Now we are ready to prove part (1). Since $\tau_A$ has
exponential moment, it follows that
$\P(\tau_A> n)=o(\P(|X|>\lambda_n)$. Therefore we may prove the result
on the event $\{\tau_A\le n\}$. We write for simplicity $\P_n(\cdot)=\P(\cdot
\cap\{\tau_A\le n\})$.
In view of Lemma~\ref{prel:lem} and 
the decomposition \eqref{eq:sumdecomp} of $S_n$ we may neglect 
the sums over the first and last cycles and it suffices to prove the \ld\ principle for the 
process   $\sum_{t=1}^{N_A(n)}S(t)$ over independent cycles. Observe that
\beao
\frac{\P_n(\lambda^{-1}_n\sum_{t=1}^{N_A(n)}S(t)\in\cdot)}{n\P(|X|\ge
  \lambda_n)}=\frac{\P_n(\lambda^{-1}_n\sum_{t=1}^{N_A(n)}S(t)\in\cdot)}{n\P(|S(1)|\ge \lambda_n)}\frac{\P(|S(1)|\ge \lambda_n)}{\P(|X |\ge \lambda_n)}.
\eeao
The same arguments as in the proof of  Lemma 4.12 in
\cite{mikosch:wintenberger:2011} (here the conditions $\la_n\to\infty$
and $\la_n/n^{\delta +1/\alpha}\to \infty$ for some $\delta>0$ are
crucial) show that for any small $\xi,\vep>0$, and any set $B$ bounded
away from zero,
\beao
\dfrac{(1-\vep)\P  (
\lambda_n^{-1}(1+\xi)^{-1}(1+\vep)^{-1}S(1)\in B )}{\E(\tau_A)\,\P(|S(1)|>\lambda_n)}
&\le& \dfrac{\P_n \Big(\lambda_n^{-1} \sum_{t=1}^{N_A(n)} S(t)\in B\Big)}{n\P(|S(1)|>\lambda_n)}+o(1)\\
& \le&
\dfrac{\P( \lambda_n^{-1}(1-\xi)^{-1}S(1)\in B)}{\E(\tau_A)\,\P(|S(1)|>\lambda_n)}+o(1)\,.
\eeao
Assume first that the cluster index $b$ does not vanish.
In view of part (2) of Theorem~\ref{main} we know that $S(1)$ is \regvary\
with index $\alpha$ and spectral \ms\ $\P_{\Theta'}$ given by \eqref{eq:seopc}, and we also know that
\beao
\dfrac{\P(\la_n^{-1} S(1)\in \cdot)}{\P(|S(1)|>\la_n)}\stv \mu_{S(1)}(\cdot)\,,
\eeao
for a non-null Radon \ms\ $\mu_{S(1)}$.
Hence, letting $\vep\to 0$ and $\xi\to 0$, 
we conclude that
$$
 \dfrac{\P _n\Big(\lambda_n^{-1} \sum_{t=1}^{N_A(n)} S(t)\in\cdot \Big)}{n\P(|S(1)|>\lambda_n)}\stv\frac{\mu_{S(1)}(\cdot)}{\E(\tau_A)}.
$$
It remains to determine the limit of $\P(|S(1)|>x)/\P(|X|>x)$ as
$\xto$.
By virtue of the proof of Theorem~\ref{main},
$a_n^{-1}\sum_{t=1}^{n/\E \tau_A} S(t)\std \xi_\alpha$. Then
necessarily
\beao
\dfrac{n}{\E \tau_A} \P(a_n^{-1} S(1)\in \cdot)\stv \nu_\alpha(\cdot)\,, 
\eeao 
where $\nu_\alpha$ is the \levy\ \ms\ of $\xi_\alpha$. Hence
\beao
\dfrac{\P(|S(1)|>a_n)}{\P(|X|>a_n)}
\sim n\P(|S(1)|\ge a_n)\to \E \tau_A \,\Gamma_\alpha (\bbs^{d-1})\,,
\eeao
where $\Gamma_\alpha$ is the spectral \ms\ of $\nu_\alpha$.
But from part (2) of Theorem \ref{main} we know that $n\P(|S(1)|\ge
a_n)\to \E\tau_A \int_{\bbs^{d-1}} b(\theta)\,dP_\Theta(\theta)$. This
proves the result in the non-generate case $b\ne 0$.
\par
In the degenerate case $b=0$,
$\P(|S(1)|\ge x)=o(\P(|X|\ge x))$ as $x\to \infty$. 
By independence of the cycles and since $\la_n/a_n\to\infty$, for any set $B$ bounded away from zero,
some $\gamma>0$,
\beao
\P_n \Big(\lambda_n^{-1} \sum_{t=1}^{N_A(n)} S(t)\in B\Big)\le \P\Big(
\Big|\sum_{t=1}^{N_A(n)} S(t)\Big|>\gamma \la_n\Big)\le
n\,c\,\P(|S(1)|\ge c\,\lambda_n)=o(n \P(|X|>a_n))=o(1)\,.
\eeao
The desired result in the degenerate case follows.
\subsection{Proof of part (2): The case $\alpha>2$.}\label{sec:largealpha}
We only consider the  non-degenerate case $b\neq0$.  We will apply Theorem~4.6 in 
\cite{mikosch:wintenberger:2011} for \fct s of \MC s in the case
$d=1$. 
\bth\label{ldMW}
Let $(G_t)=(f(\Phi_t))$ be a 1-dimensional \fct al of 
a strictly stationary $\bbr$-valued irreducible
aperiodic \MC\ $(\Phi_t)$ . Write $S_n(G)=G_1+\cdots +G_n$, $n\ge 1$, for the
corresponding random walk. Assume that the following conditions are satisfied. 
\begin{enumerate}
\item
The \regvar\ condition ${\bf (RV_\alpha)}$  for some $\alpha>2$ and $\E G=0$.
\item
The anti-clustering condition $({\bf AC})_\alpha$:
\beao
\lim_{k\to\infty}\limsup_{\nto} \sup_{x\in \Lambda_n}
\delta_k^{-\alpha}\sum_{j= k}^n\P(|G_j|> x\delta_k\mid |G_0|> x\delta_k)=0\,.
\eeao
for a \seq\
$\delta_k=o(k^{-2})$, $\kto$, and sets $(\Lambda_n)$ \st\ $b_n=\inf
\Lambda_n\to \infty$ as $\nto$.
\item
The limit
$
b_+=\lim_{k\to\infty}(b_+(k+1)-b_+(k))
$ exists, where the constants $(b_+(k))$ are defined in Theorem~\ref{thm:clt}.
\item
The drift condition  ${\bf (DC_{\it p})}$ for every $p<\alpha$.
\end{enumerate}
Then the precise \ld\ principle
\beam\label{eq:unif}
\lim_{n\to \infty}\sup_{x\in \Lambda_n}\Big|\frac{\P(S_n(G)> x)}{n\,\P(|G|> x)}-b_+\Big|=0\,,
\eeam
holds if $\Lambda=(b_n,c_n)$ for any  \seq\ $(b_n)$
satisfying $b_n=n^{0.5+\vep}$ for  any $\vep>0$,
and $(c_n)$ \st\ $c_n>b_n$ and 
\beam\label{eq:tau}
\P(\tau_A>n)= o(n\, \P(|G|>c_n))\,,
\eeam
where $\tau_A= \tau_A(1)$ is the first hitting time of the atom $A$ of
the \MC ; see Section~\ref{subsec:drift}.
\ethe
We will apply this result to  $G_t=\theta'X_t$, $t\in\bbz$,
for any fixed $\theta\in\bbs^{d-1}$ with $b(\theta)\ne 0$.
Note that \eqref{eq:tau} is satisfied since  $\tau_A$ has exponential
moment. Condition ${\bf (RV_\alpha)}$ for $(G_t)$  
is satisfied by \regvar\ of $(X_t)$ in all non-degenerate cases 
$b(\theta) \ne 0$. {The
existence of the limits
$b_+=\lim_{k\to\infty}(b_+(k+1)-b_+(k))=
b(\theta)/\E|\theta '\Theta_0|^\alpha$ (here we assume that 
$\E|\theta '\Theta_0|^\alpha\ne 0$)} is ensured by Theorem \ref{lem:sum}. 
It remains to check condition $({\bf AC})_\alpha$ for $(G_t)$ under  ${\bf (DC_{\it p})}$ for $(G_t)$ for every $p<\alpha$. 
Note that
 ${\bf (DC_{\it p})}$ for $(X_t)$ implies  ${\bf (DC_{\it p})}$ for $(G_t)$. 
Using Markov's inequality of order
$p<\alpha$ and \eqref{rec:dcp}, we obtain the following bound for $k\ge 1$, $x\in \Lambda_n$:
\beao
\sum_{j= k}^n\P(|G_j|> x \delta_k \mid |G_0|> x\delta_k)&\le& 
\sum_{j= k}^{n}\frac{\E(|G_j|^p\1_{\{|G_0|>x \delta_k\}})}{x^p\delta_k^p
\P(|G_0|>x\delta_k)}\\&\le&\sum_{j= k}^{n}\Big(\frac{\beta^{j-1}\E(|X_0|^p\1_{\{|X_0|>x\delta_k\}})}{x^p\delta_k^p\P(|G_0|>x\delta_k)}+\frac{c}{x^p\delta_k^p}\Big)\\
&\le&c\Big(\frac{\beta^k\E(|X_0|^p\1_{\{|X_0|>x\delta_k\}})}{x^p\delta_k^p\P(|X_0|>x\delta_k)}+\frac{n}{x^p\delta_k^p}\Big).
\eeao
The second term is of the order $O(n b_n^{-p})=o(1)$ uniformly
for $x\in\Lambda_n$ since  $p$ can be chosen larger than 2 
such that $p(0.5+\vep)>1$. 
The first term converges to $c\beta^k$ as $\nto$ uniformly for 
$x\in \Lambda_n$, by applications of Karamata's Theorem and the
uniform \con\ theorem of \regvar .
We conclude that $({\bf AC})_\alpha$ holds as $\delta_k^{-1}\beta^k\to
0$ as $\kto$ if we choose  
$\delta_k=k^{-2-\vep'}$ for $\vep'>0$ sufficiently small.
Thus all conditions of Theorem~\ref {ldMW} are satisfied for
$(G_t)=(\theta'X_t)$ and therefore \eqref{eq:unif} applies. { Since 
$\P(|\theta'X|>x)/\P(|X|>x)\to \E[ |\theta'X|^\alpha]$ we can also
write \eqref{eq:unif} in the form \eqref{eq:rig}.
\par
Now choose $(\la_n)$ as in the formulation of the theorem and apply
Lemma~\ref{lem:ldp} below. This proves the theorem.
}

\appendix
\section{}\setcounter{equation}{0}
The following result is useful for proving multivariate \ld\ results
and \clt s.
\ble\label{lem:ldp}
Assume that  $(X_t)$ is an  $\bbr^d$-valued strictly stationary \seq\
which is \regvary\ with index $\alpha>0$ and satisfies the
one-dimensional \ld\ principle
\beam\label{eq:43}
\dfrac{\P(\theta'S_n>\la_n)}{n\,\P(|X|>\la_n)}\to b(\theta)
\,,\quad \theta\in \bbs^{d-1}\,,
\eeam
for some \seq\ $\la_n\to\infty$ \st\ $n\P(|X|>\la_n)\to 0$. 
Moreover, assume that $\alpha\not\in \bbn$ or $b(\cdot)=b(-\cdot)$.
Then \eqref{eq:lda} holds.
\ele
\begin{proof}
Define the \ms s 
\beao
m_n(\cdot)=\dfrac{\P(\la_n^{-1} S_n\in \cdot)}{n\,\P(|X|>\la_n)}\,,\quad n\ge 1\,,
\eeao
on the Borel $\sigma$-field of $\ov\bbr_0^d$.
We conclude from  \eqref{eq:43} that for any Borel set $B$ bounded 
away from zero,
\beao
\sup_{n\ge 1}m_n(B)<\infty\,.
\eeao 
This means that $(m_n(B))$ is vaguely tight; see 
\cite{kallenberg:1983,resnick:1987}. In view of  \eqref{eq:43}, any vague 
subsequential limit $\mu$ of $(m_n)$ satisfies the relation \eqref{eq:cb}.
For non-integer $\alpha$, the latter property 
combined with the proof of Theorem 1.1 in 
\cite{basrak:davis:mikosch:2002} shows that all vague subsequential limits
of  $(m_n)$ are identical and uniquely determined by the 
property \eqref{eq:cb}. Hence \eqref{eq:lda} holds and the limit $\nu_\alpha$ is
given by \eqref{eq:cb}. A careful study of the proof of Theorem 1.1 in 
\cite{basrak:davis:mikosch:2002} shows that the proof remains valid
if the  subsequential limits
have the property $\mu(\cdot)=\mu(-\cdot)$ which 
follows if $b(\cdot)=b(-\cdot)$. 
\end{proof}
\noindent
\subsubsection*{Acknowledgments} We would like to thank the referee for careful
reading of our paper and for useful commments.

\end{document}